\documentclass[a4paper,11pt]{amsart}
\usepackage{amsfonts,amsthm,amsmath,amssymb,graphicx}
\theoremstyle{plain}
\newtheorem{lem}{Lemma}[section]
\newtheorem{prop}[lem]{Proposition}
\newtheorem{thm}[lem]{Theorem}
\newtheorem{cor}[lem]{Corollary}
\theoremstyle{definition}

\theoremstyle{remark}

\DeclareMathOperator{\sgn}{sgn}{\Large {\normalsize }}
\DeclareMathOperator{\cls}{cls}
\DeclareMathOperator{\rank}{rank}

\DeclareMathOperator{\mass}{mass}
\DeclareMathOperator{\ord}{ord}
\DeclareMathOperator{\sym}{sym}

\DeclareMathOperator{\e}{e}
\DeclareMathOperator{\mult}{mult}
\DeclareMathOperator{\diag}{diag}
\DeclareMathOperator{\gen}{gen}
\DeclareMathOperator{\Tr}{tr}

\newcommand{\bmu}{\boldsymbol \mu}

\newcommand{\bdelta}{\boldsymbol \delta}
\newcommand{\bbeta}{\boldsymbol \beta}
\newcommand{\balpha}{\boldsymbol \alpha}
\newcommand{\boldeta}{\boldsymbol \eta}

\newcommand{\calP}{{\mathfrak P}}
\newcommand{\I}{{\mathcal I}}

\newcommand{\kfld}{{\mathbb K}}

\newcommand{\Z}{\mathbb Z}
\newcommand{\Q}{\mathbb Q}

\newcommand{\F}{\mathbb F}

\newcommand{\R}{\mathbb R}
\newcommand{\C}{\mathbb C}
\newcommand{\stufe}{\mathcal N}

\newcommand{\h}{\mathfrak H}

\newcommand{\K}{\mathcal K}
\newcommand{\M}{\mathcal M}
\newcommand{\Y}{\mathcal Y}

\newcommand{\calL}{\mathcal L}
\newcommand{\calU}{\mathcal U}

\newcommand{\calQ}{\mathfrak Q}
\newcommand{\frakD}{\mathfrak D}
\newcommand{\n}{\mathfrak N}
\newcommand{\Ok}{\mathcal O}

\newcommand{\q}{\mathfrak q}
\newcommand{\calA}{\mathcal A}
\newcommand{\scale}{\text{scale}}
\newcommand{\norm}{\text{norm}}

\begin{document}

% Enter full title and short title for running headers
\title[Hecke operators on Hilbert-Siegel theta series]
{Hecke operators on Hilbert-Siegel theta series}

% Author name(s)
\author{Dan Fretwell}
\address{School of Mathematics, University of Bristol, University Walk, Clifton, Bristol BS8 1TW, United Kingdom}
\email{dan.fretwell@bristol.ac.uk}
\author{Lynne Walling}
\address{School of Mathematics, University of Bristol, University Walk, Clifton, Bristol BS8 1TW, United Kingdom;
phone +44 (0)117 331-5245}
\email{l.walling@bristol.ac.uk}

% Abbreviated author name for running headers
%\abbrevauthor{L.H. Walling}
% Abbreviated author name for first page header
%\headabbrevauthor{Walling, L.H.}


\begin{abstract} 
We consider the action of Hecke-type operators on Hilbert-Siegel theta series attached to lattices of even rank.  We show that average Hilbert-Siegel theta series are eigenforms for these operators, and we explicitly compute the eigenvalues.
\end{abstract}

\maketitle
\def\thefootnote{}
\def\thefootnote{\arabic{footnote}}

\section{Introduction}
Siegel theta series help us study quadratic forms on lattices, as their Fourier coefficients carry information about the structures of their sublattices.  Hecke operators help us study Fourier coefficients of modular forms.  Here we consider Hecke-type operators on Hilbert-Siegel theta series, showing that the average Hilbert-Siegel theta series is an eigenform for these operators, and we explicitly compute the eigenvalues.

With $\kfld$ a totally real number field with ring of integers $\Ok$ and $L$ a lattice of rank $m$ over $\Ok$, we do not know that $L$ is a free $\Ok$-module, but we do have
$$L=\calA_1 x_1\oplus\cdots\oplus\calA_m x_m$$
where $\calA_1,\ldots,\calA_m$ are (nonzero) fractional ideals and $x_1,\ldots,x_m$ are vectors in the space
$\kfld L=\kfld x_1\oplus\cdots\oplus\kfld x_m$.  We equip $L$  with a totally positive quadratic form $\q$ (so $\q(x)\gg0$ for all $x\in L$).  To build a degree $n$ Hilbert-Siegel theta series associated to $L$, we set
$\calL=\big<\calA_1,\ldots,\calA_m\big>\Ok^{m,n}$ where
$\big<\calA_1,\ldots,\calA_m\big>$ is shorthand for $\diag\{\calA_1,\ldots,\calA_m\}.$
We set $Q=\big(B_{\q}(x_i,x_j)\big)$ where $B_{\q}$ is the symmetric bilinear form associated to $\q$ so that $\q(x)=B_{\q}(x,x)$.  Then we set
$$\theta(L;\tau)=\theta(\calL;\tau)=\sum_{U\in\calL}\e\{Q[U]\tau\}$$
where $Q[U]=\,^tUQU$, $\tau$ is a suitable complex variable, and $\e\{*\}$ is a suitable exponential function (defined below).  Note that with $U\in\calL$ and $(y_1\,\cdots\,y_n)=(x_1\,\cdots\,x_m)U$, each $y_i$ lies in $L$ and $\big(B_{\q}(y_i,y_j)\big)=Q[U]$; so the Fourier coefficients of $\theta(L;\tau)$
tell us how often $L$ ``represents" any given quadratic form $T$ of dimension $n$, meaning the number of sublattices of $L$ that inherit $T$ as a quadratic form.

The goal of this paper is to show that when $m$ is even with $m=2k$,
 the average Hilbert-Siegel theta series $\theta(\gen L;\tau)$ (defined in Section 5)
is an eigenform for certain Hecke-type operators, yielding relations on average representation numbers of the quadratic form $\q$ on $L$.  This generalizes what was done in \cite{theta I} and \cite{theta II} where the number field was $\mathbb Q$; many of the arguments in this previous work are local, and those arguments generalize to the number field setting quite easily.  The main work in this paper is to realize the action of the Hecke-type operators on $\theta(L;\tau)$ in terms of $L$; this involves the use and understanding of auxiliary theta functions, defined as follows.

For (nonzero) fractional ideals $\I_1,\ldots,\I_n$, we have the lattice
$$\calL'=\calL\big<\I_1,\ldots,\I_n\big>
=\big<\calA_1,\ldots,\calA_m\big>\Ok^{m,n}\big<\I_1,\ldots,\I_n\big>$$
(so for $U'\in\calL'$ and $(z_1\,\cdots\,z_n)=(x_1\,\cdots\,x_m)U'$, we have $z_{\ell}\in\I_{\ell}L$ for $1\le\ell\le n$).  We set
$$\theta(\calL';\tau)=\sum_{U'\in\calL'}\e\{Q[U']\tau\}.$$
In Section 2, we first extend an algebraic trick of Eichler \cite{Eichler} to establish an Inversion Formula
(Theorem 2.4) for 
a very general theta series.
From hereon, we assume $m$ is even with $m=2k$; then
with $\calL'$ as above, variations on standard techniques show that $\theta(\calL';\tau)$ is a Siegel modular form of degree $n$ and weight $k$, for a group of some level $\stufe$ and character $\chi$ (defined below; note that $\stufe$ and $\chi$ are determined by $L$, and $\theta(\calL';\tau)$ transforms under a group determined by $\I_1,\ldots,\I_n$ as well as invariants of $L$).

In Section 3 we show how $\theta(\calL;\tau)$ and $\theta(\calL';\tau)$ are related through maps $S_{\ell}(\calQ)$ ($\calQ$ a fractional ideal); proving this requires the generalized Inversion Formula (Theorem 2.4).
In Section 4 we evaluate
$$\theta(L;\tau)|S_1(\calP^{-1})\cdots S_j(\calP^{-1})T_j(\calP^2)$$
where $\calP$ is a prime ideal not dividing the level of $L$ and $T_j(\calP^2)$ is a Hecke operator
with $1\le j\le n$ (defined in Section 4); we consider this image of $\theta(L;\tau)$ as it lies in the same space as $\theta(L;\tau)$.
We first describe $\theta(L;\tau)|S_1(\calP^{-1})\cdots S_j(\calP^{-1})T_j(\calP^2)$ in terms of lattices
$$\calL_{r_0,r_2}=\calL\big<\calP^{-1} I_{r_0},I_{n-r_0-r_2},\calP I_{r_2}\big>$$
with $r_0,r_2$ varying so that $r_0+r_2\le j$
(Proposition 4.1).  We then partition each $\calL_{r_0,r_2}$ into equivalence classes, and partition each equivalence class into orbits of a group $\K_{r_0,r_2}$ 
(defined in Section 4) 
to get a computationally convenient description of $\theta(\calL_{r_0,r_2};\tau)$ (Proposition 4.2).
Then we can write $\theta(L;\tau)|S_1(\calP^{-1})\cdots S_j(\calP^{-1})T_j(\calP^2)$ as a sum over lattices $\Omega\subseteq\calP^{-1} L$ (Proposition 4.4).  As this sum involves incomplete character sums, we complete them in Theorem 4.5 by replacing the operator
$S_1(\calP^{-1})\cdots S_j(\calP^{-1})T_j(\calP^2)$
by $\widetilde T_j(\calP^2)$,
 a linear combination of the operators
 $S_1(\calP^{-1})\cdots S_{\ell}(\calP^{-1})T_{\ell}(\calP^2)$, $0\le\ell\le j$.
Then in Section 5 we appeal to the (primarily local) arguments in \cite{theta I} and \cite{theta II} that carry over almost directly to the number field setting.  With $\{T'_j(\calP^2):\ 0\le j\le n\ \}$ a specific generating set for
the algebra generated by  $\{\widetilde T_j(\calP^2):\ 0\le j\le n\ \}$
(where $T'_0(\calP^2)=\widetilde T_0(\calP^2)$ is the identity map), we show that for $\calP$
a prime ideal not dividing the level of $L$,
$$\theta(\gen L;\tau)|T'_j(\calP^2)=\lambda_j(\calP^2)\theta(\gen L;\tau)$$
where
$$\lambda_j(\calP^2)=
\begin{cases}
N(\calP)^{j(k-n)+j(j-1)/2}\bbeta(n,j)\bdelta(k-1,j)
&\text{if $\chi(\calP)=1$ and $j\le k$,}\\
N(\calP)^{j(k-n)+j(j-1)/2}\bbeta(n,j)\bmu(k-1,j)
&\text{if $\chi(\calP)=-1$ and $j<k$,}\\
0&\text{otherwise}\end{cases}$$
(Corollary 5.4).
Here $\bbeta(n,j)$ is the number of $j$-dimensional subspaces of an $n$-dimensional space over $\Ok/\calP$,
and
\begin{align*}
\bdelta(k-1,j)&=(N(\calP)^{k-1}+1)\cdots(N(\calP)^{k-j}+1),\\
\bmu(k-1,j)&=(N(\calP)^{k-1}-1)\cdots(N(\calP)^{k-j}-1).
\end{align*}

Note that while herein we assume $m$ is even, in a subsequent paper we demonstrate how to adjust these computations to allow $m$ odd.

\bigskip
\noindent{\bf Notation.}  
Throughout, we take $\kfld$ to be a totally real number field with ring of integers $\Ok$
and different $\partial$.
Except for Sections 2.1 and 2.2,
we take $m\in\Z_+$ to be even with $m=2k$, and we fix fractional ideals $\calA_1,\ldots,\calA_m$ and vectors $x_1,\ldots,x_m$ so that $L=\calA_1x_1\oplus\cdots\oplus\calA_mx_m$ is a fixed lattice.
We fix $n\in\Z_+$ and set
$$\calL=\big<\calA_1,\ldots,\calA_m\big>\Ok^{m,n}.$$
We take $\q$ to be a totally positive quadratic form on $L$ whose associated symmetric bilinear form $B_{\q}$ satisfies $\q(x)=B_{\q}(x,x)$; we set
$Q=\big(B_{\q}(x_i,x_j)\big)$.

When $L'$ is a free lattice with quadratic form $\q'$ given by a matrix $Q'$ (relative to some $\Ok$-basis for $L'$),
we write $L'\simeq Q'$, and we say $L'$ is isometric to $Q'$.
The discriminant of $L'$ is $\det Q'$, which is well-defined up to squares of units in $\Ok$.

The complement of $L$ is
$$\widetilde L=\{v\in\kfld L:\ \Tr(B_{\q}(v,L))\subseteq\Z\ \}$$
where $tr$ denotes the trace from $\kfld$ to $\Q$.
So with $(x_1'\,\cdots\,x_m')=(x_1\,\cdots\,x_m)Q^{-1}$,
$\widetilde L=\partial^{-1}(\calA_1^{-1}x_1'\oplus\cdots\oplus\calA_m^{-1}x_m').$
We define $\norm L$ to be the fractional ideal generated by $\{\q(x):\ x\in L\ \}$,
and we define $\scale L$ to be the fractional ideal generated by $\{B_{\q}(x,y):\ x,y\in L\ \}$.
Note that $2\scale L\subseteq \norm L\subseteq\scale L$.
With $\n=\n(L)=\frac{1}{2}\norm L,$ we say $L$ is even $\n$-integral, meaning that for all $x,y\in L$ we have $\q(x)\in 2\n$ and $B_{\q}(x,y)\in\n$.
We define the level of $L$ to be
$$\stufe=4(\norm L\cdot\norm L^{\#})^{-1},$$
which is an integral ideal.  For $\calP$ a prime ideal, we have $\calP\nmid\stufe$ if and only if $\Ok_{\calP}L$ is $\Ok_{\calP}\n$-modular, meaning that with $\eta',\alpha_i\in\kfld$ so that
$\eta'\Ok_{\calP}=\n^{-1}\Ok_{\calP}$ and $\alpha_i\Ok_{\calP}=\calA_i\Ok_{\calP}$ ($1\le i\le m$), 
the matrix $\eta' Q[\big<\alpha_1,\ldots,\alpha_m\big>]$ is invertible over $\Ok_{\calP}$.
(Note that when $\calP|2\Ok$, the matrix $\eta' Q[\big<\alpha_1,\ldots,\alpha_m\big>]$ necessarily has diagonal entries in $2\Ok_{\calP}$.)

Let $\h_{(n)}=\{X+iY:\ X,Y\in\R^{n,n}_{\sym},\ Y>0\ \}$, where $\R^{n,n}_{\sym}$ denotes the set of symmetric $n\times n$ matrices over $\R$, and $Y>0$ denotes that as a quadratic form, $Y$ is positive definite.  With $d$ the strict ideal class number of $\kfld$ and $\psi_1,\ldots,\psi_d$ the embeddings of $\kfld$ into $\R$, for $\tau=(\tau_1,\ldots,\tau_h)\in\h_{(n)}^d$ and $M\in\kfld^{n,n}_{\sym}$, we set $M^{(i)}=\psi_i(M)$,
$$\Tr(M\tau)=\sum_{i=1}^d M^{(i)}\tau_i,$$
and we set $\e\{M\tau\}=\exp(\pi i\sigma(\Tr(M\tau)))$ where $\sigma$ denotes the trace of a matrix.

In Section 2 we will see that with $\I_1,\ldots,\I_n$ fractional ideals and
$\calL'=\calL\big<\I_1,\ldots,\I_n\big>$, the theta series
$$\theta(\calL';\tau)=\sum_{U'\in\calL'}\e\{Q[U']\tau\}$$
is a modular form for the group $\Gamma_0(\stufe;\I_1,\ldots,\I_n;\n)$ defined as follows.
Set $X=\big<\I_1,\ldots\I_n\big>$, $X^{-1}=\big<\I_1^{-1},\ldots\I_n^{-1}\big>$, and set
\begin{align*}
&\Gamma_0(\stufe;\I_1,\ldots,\I_n;\n)
=\Gamma^{(n)}_0(\stufe;\I_1,\ldots,\I_n;\n)\\
&\quad=
\begin{pmatrix}(\n\partial X)^{-1}\\&X\end{pmatrix}
\begin{pmatrix}\Ok^{n,n}&\Ok^{n,n}\\ \stufe\Ok^{n,n}&\Ok^{n,n}\end{pmatrix}
\begin{pmatrix}\n\partial X\\&X^{-1}\end{pmatrix}
\cap Sp_n(\kfld).
\end{align*}
(Here $Sp_n(\kfld)$ denotes the group of $2n\times 2n$ symplectic matrices over $\kfld$.)
For $\chi$ a character modulo $\stufe$, we write
$\M_k(\Gamma_0(\stufe;\I_1,\ldots,\I_n;\n),\chi)$ to denote the space of
(Hilbert-Siegel) modular forms for the group
$\Gamma_0(\stufe;\I_1,\ldots,\I_n;\n)$ with character $\chi$, meaning the set of
 analytic functions from $\h_{(n)}^d$ to $\C$ so that for all
$\gamma=\begin{pmatrix}A&B\\C&D\end{pmatrix}\in\Gamma_0(\stufe;\I_1,\ldots,\I_n;\n)$,
$$f|\gamma(\tau)=\chi(\det D) f(\tau)$$
where $f|\gamma(\tau)=\det(N(C\tau+D))^{-k} f((A\tau+B)(C\tau+D)^{-1})$
and
$N(C\tau+D)=\prod_{i=1}^d (C^{(i)}\tau_i+D^{(i)}).$
We will also see that the character associated to $\theta(\calL';\tau)$ as a modular form is
$\chi_L$, defined as follows.

First take $\calP$ to be a prime ideal with $\calP\nmid2\Ok$.  
By Section 92 \cite{O'M}, a binary unimodular $\Ok_{\calP}$-lattice is isometric to either a hyperbolic plane (given by the matrix $\big<1,-1\big>$) or an anisotropic plane (given by the matrix $\big<1,-\omega\big>$ where $\omega$ is a non-square unit in $\Ok_{\calP}$); a dimension $2k$ unimodular $\Ok_{\calP}$-lattice is isometric to either an orthogonal sum of $k$ hyperbolic planes, or an orthogonal sum of $k-1$ hyperbolic planes and an anisotropic plane.  Thus the Lengendre-type symbol $\left(\frac{*}{\calP}\right)$ given by
$$\left(\frac{y}{\calP}\right)=
\begin{cases}1&\text{if $y$ is a square modulo $\calP$,}\\
-1&\text{otherwise}
\end{cases}$$
(where $y\in\Ok\smallsetminus\calP$) allows us to distinguish these two types of dimension $2k$ 
lattices over $\Ok_{\calP}$.

Next take $\calP$ to be a prime ideal with $\calP|2\Ok$; let $\calU=\Ok_{\calP}^{\times}$.
By 93:11 \cite{O'M}, a binary even unimodular lattice over $\Ok_{\calP}$ is either a hyperbolic plane (given by the matrix $\begin{pmatrix}0&1\\1&0\end{pmatrix}$) or an even anisotropic plane (given by the matrix
$\begin{pmatrix}2&1\\1&2\epsilon\end{pmatrix}$ where $1+4\epsilon$ has quadratic defect $4\Ok_{\calP}$;
note that by Section 63A \cite{O'M}, $1-4\epsilon$ also has quadratic defect $4\Ok_{\calP}$, so in particular,
$1-4\epsilon$ is not a square, and the set of units in $\Ok_{\calP}$ with quadratic defect $4\Ok$
is $(1-4\epsilon)\calU^2$).
By 93:18(ii) \cite{O'M}, 
a dimension $2k$ even unimodular $\Ok_{\calP}$-lattice is isometric to either an orthogonal sum of $k$ hyperbolic planes, or an orthogonal sum of $k-1$ hyperbolic planes and an even anisotropic plane.
We want to choose a quadratic character modulo some power of $\calP$
that will distinguish a hyperbolic plane and an even anisotropic plane over $\Ok_{\calP}$.  
So we want to construct a homomorphism $\varphi:\calU/\calU^2\to \{\pm1\}$ so that
$\varphi((1-4\epsilon)\calU^2)=-1.$
By 63:9 \cite{O'M}, $|\calU/\calU^2|=2([\Ok_{\calP}:\calP\Ok_{\calP}])^{\ord_{\calP}(2)}$, so $\calU/\calU^2$ is a 2-group with at least 4 elements, in which every element is its own inverse.  Thus with 
$|\calU/\calU^2|=2^d$, we can choose $(1-4\epsilon)\calU^2$ together with $d-1$ other elements to generate $\calU/\calU^2$, and hence we can build $2^{d-1}$ homomorphisms $\varphi$ that map 
$(1-4\epsilon)\calU^2$ to $-1$.  Such a homomorphism $\varphi$ corresponds to a choice of a
Kronecker-type symbol
that (by 63:1 \cite{O'M}) is a quadratic character modulo $4\calP$ with
$$\left(\frac{y}{\calP}\right)_{\varphi}=\varphi(y\calU^2)$$
($y\in\calU$).
We fix a choice of $\varphi$ and write $\left(\frac{y}{\calP}\right)$ for 
$\left(\frac{y}{\calP}\right)_{\varphi}$.

Now suppose that $\calP\nmid \stufe$; fix $\eta'\in\n^{-1}$ so that $\eta'\Ok_{\calP}=\n^{-1}\Ok_{\calP}$, and for $1\le i\le m$, fix $\alpha_i\in\calA_i$ so that $\alpha_i\Ok_{\calP}=\calA_i\Ok_{\calP}$. Set
$$\chi^*(\calP)=\chi^*_L(\calP)
=\left(\frac{(-1)^k(\eta')^{m}(\alpha_1\cdots\alpha_m)^2\det Q}{\calP}\right);$$
note that by Section 93 of \cite{O'M}, when $\calP|2\Ok$ (and $\calP\nmid\stufe$) we have that
$$(-1)^k(\eta')^{m}(\alpha_1\cdots\alpha_m)^2\det Q\text{ lies in }\calU^2 \text{ or }
(1-4\epsilon)\calU^2,$$
and hence $\chi^*(\calP)$ is independent of the choice of Kronecker-type symbol chosen above.

Now we extend $\chi^*$ completely multiplicatively to all fractional ideals $\I$ that are relatively prime to $\stufe$ (where $\I$ relatively prime to $\stufe$ means that for any prime $\calP'|\stufe$, $\ord_{\calP'}\I=0$).  Then for $y\in\Ok$ with $y\Ok$ relatively prime to $\stufe$, we set
$$\chi(y)=\chi_L(y)=(N(\sgn y))^k\chi^*(y\Ok)$$
where $N(\sgn y)=\prod_{i=1}^d sgn(\psi_i(y)).$
% Note that this means $\chi_L(-1)=?$.

Note that for $\gamma\in\Gamma_0(\stufe;\I_1,\dots,\I_n;\n)$ we sometimes write $\chi_L(\gamma)$ to denote $\chi_L(\det D_{\gamma})$.

We define the maps $S_{\ell}(\calQ)$ in Section 3, and the maps $T_j(\calP^2)$ in Section 4.

\bigskip

\section{Hilbert-Siegel theta series}
In this section we will introduce Hilbert-Siegel theta series and prove the transformation formula for such series. In order to do this we will first need an inversion formula. We prove this by generalizing a 
%well known 
trick of Eichler \cite{Eichler}, used in the Hilbert modular case.
Note that we will be realizing a Hilbert-Siegel theta series as a Siegel theta series over the rationals, and so by Theorem 1.1.4 \cite{And}, we know the the Hilbert-Siegel theta series is analytic.

\begin{subsection}{Inversion formula: local notation and setup.}

Let $\gamma_1,\ldots,\gamma_d$ be an integral basis for $\Ok$ (meaning that $\Ok=\Z\gamma_1+\cdots+\Z\gamma_d$).
As alluded to above, we have fixed an ordering of the real embeddings $\psi_1,...,\psi_d: \kfld\hookrightarrow\R$
and we write 
%. The conjugates of $\alpha\in \kfld$ will be written as 
$\alpha^{(w)}$ to denote $\psi_w(\alpha)$. 
The maps $\psi_w$ induce linear embeddings $\kfld^{a,b}\hookrightarrow \R^{a,b}$ for any $a,b\geq 1$ by applying $\psi_w$ component-wise. For ease of notation we write $U^{(w)} = \psi_w(U)$ for $U\in \kfld^{a,b}$ and refer to these as the conjugates of $U$. For $U\in \kfld^{a,b}$ we define $\Tr(U) = \sum_{w=1}^dU^{(w)}$ (i.e. apply the field trace on each entry). 

Fix $m,n\geq 1$ and set $V = \kfld^{m,n}$. 
Let $M\in\kfld^{m,m}_{\sym}$ be a totally positive definite matrix, corresponding to a quadratic form 
$\q'$ on $\kfld^m$. This also gives rise to a symmetric bilinear form $B_{\q'}$ on $\kfld^m$ 
(taking the convention that $B_{\q'}(x,x) = \q'(x)$).
After scaling we may assume that $M$ is integral with even diagonal. 
Given $U,W\in V$ we write $M[U,W] = {^t}UMW\in \kfld^{n,n}_{\sym}$ and $M[U] = {^t}UMU\in \kfld^{n,n}$. 
%Note that $M[U]$ is a symmetric matrix. 
Then considering $V$ as both a $\kfld$-vector space and a $\Q$-vector space gives quadratic forms $M_{V,\kfld}, M_{V,\Q}$ respectively, given by $M_{V,\kfld}(U) = \sigma(M[U])$ and $M_{V,\Q}(U) = \sigma(\Tr(M[U]))$, where
we recall that $\sigma$ is the matrix trace. The corresponding symmetric bilinear forms are $B_{V,\kfld}(U,W) = \sigma(M[U,W])$ and $B_{V,\Q}(U,W) = \sigma(\Tr(M[U,W]))$.

Now consider an $\Ok$-lattice $\calL'\subseteq V$ of rank $mn$
(where $m$ is not necessarily even). We do not assume that $\calL'$ is free. However restricting scalars to $\Z$ we must obtain a free $\Z$-lattice of rank $mnd$. Let $U_1, U_2, ... , U_{mnd}$ be a $\Z$-basis for $\calL'$. We may restrict the above quadratic forms to $\calL'$, writing $M_{\calL',\kfld}$ and $M_{\calL',\Q}$ to emphasize this.
%
%Let $\tau = (\tau^{(1)},...,\tau^{(d)})\in\HH_n^d$ (Siegel upper half space of degree $n$). Abusing %notation slightly we will write: \[Q_{V,\Q}(U\tau) = \sigma(\Tr(Q[U]\tau)) = \sigma\left(\sum_{w=1}%^dQ[U]^{(w)}\tau^{(w)}\right).\] 
%
Attached to $\calL'$ is a theta series: \[\theta(\calL';\tau) = \sum_{U\in\calL'}e^{\pi i M_{\calL,\Q}(U\tau)} = \sum_{U\in\calL'} e\{M_{\calL',\kfld}[U]\tau\}.\]
We will prove an inversion formula for this very general theta series; then we restrict our 
attention to the lattices described in the introduction, and prove that these are Hilbert-Siegel modular forms.
Note that since $\theta(\calL';\tau)$ gets identified as a Siegel theta series, we know from Theorem 1.1.4 \cite{And} that $\theta(\calL';\tau)$ is analytic.

% We will show that $\theta(\calL';\tau)$ is a Hilbert-Siegel modular form %and make precise the level, weight 
%and character. Later we will restrict to a specific family of $\calL'$ and %study the action of Hecke operators %on $\theta(\calL';\tau)$.

\end{subsection}

\begin{subsection}{The inversion formula for $\theta(\calL';\tau)$.}

To relate $\theta(\calL';\tau)$ to $\theta(\widetilde{\calL'},-\tau^{-1}),$
we first consider the quadratic form $M_{V,\Q}$ in more detail. We will show that there is a matrix $Z_1(\tau)\in \h_{(mnd)}$ and an isomorphism $\phi':V \rightarrow\Q^{mnd}$ satisfying: \[M_{V,\Q}[U]\tau = {^t}\phi'(U) Z_1(\tau)\phi'(U),\] i.e. that $M_{V,\Q}$ is isometric to the quadratic form on $\Q^{mnd,1}$ with Gram matrix $Z_1(I)$. For $n=1$ this is precisely the idea behind Eichler's trick for establishing the inversion formula for Hilbert theta series.

Before stating the result we must make a few definitions. First we construct the matrix $G = (\gamma_j^{(i)})\in \R^{d,d}$. Given $G$ we then form the Kronecker product $G' = I_n \otimes G \otimes I_m\in \R^{mnd,mnd}$.

Consider the linear map $\phi: \kfld^{m,1} \rightarrow \Q^{md,1}$ that sends column vector \[\textbf{u} = {^t}(u_1,...,u_m) = \sum_{i=1}^m u_i \textbf{e}_i = \sum_{i,j} u_{i,j}\textbf{e}_i \gamma_j\] to \[(u_{1,1},u_{2,1},...,u_{m,1},u_{1,2},...,u_{m,2},...,u_{m,d})\in\Q^{md}.\] This map extends to a linear map $\phi: V \rightarrow \Q^{mnd}$ via $\phi(U) = (\phi(\textbf{u}_1),...,\phi(\textbf{u}_n))$, where $\textbf{u}_1,...,\textbf{u}_n$ are the columns of $U$. Let $\phi'(U) = {^t}\phi(U)$.

We see that there is a strong link between the conjugates of $U\in V$ and the vector $G'\phi'(U)\in\Q^{mnd,1}$.

\begin{lem}
For any $U\in V$ with columns $\emph{\textbf{u}}_1,...,\emph{\textbf{u}}_n$ we have \[G'\phi'(U) = {^t}(\emph{\textbf{v}}_1^{(1)},\emph{\textbf{v}}_1^{(2)},...,\emph{\textbf{v}}_1^{(d)},\emph{\textbf{v}}_2^{(1)},...,\emph{\textbf{v}}_2^{(d)},...,\emph{\textbf{v}}_n^{(d)}),\] where $\emph{\textbf{v}}_w = {^t}\emph{\textbf{u}}_w$.
\end{lem}

\begin{proof}
Note that: \[G'\phi'(U) = \left(\begin{array}{c}(G\otimes I_m){^t}\phi(\textbf{u}_1)\\ (G\otimes I_m){^t}\phi(\textbf{u}_2)\\ \ldots \\ (G\otimes I_m){^t}\phi(\textbf{u}_n)\end{array}\right).\]

It suffices to show that $(G\otimes I_m){^t}\phi(\textbf{u}_w) = {^t}(\textbf{v}_w^{(1)}, ... ,\textbf{v}_w^{(d)})$ for each $1\leq w \leq n$. Letting $\phi(\textbf{u}_w) = (u_{1,1,w},u_{2,1,w},..., u_{m,1,w}, u_{1,2,w},...,u_{m,2,w},...,u_{m,d,w})$ the claim follows since: \begin{align*}(G\otimes I_m){^t}\phi(\textbf{u}_w) &= \left(\begin{array}{c}\sum_{j=1}^{d}\gamma_j^{(1)}\sum_{i=1}^m u_{i,j,w}\textbf{e}_i\\ \sum_{j=1}^{d}\gamma_j^{(2)}\sum_{i=1}^m u_{i,j,w}\textbf{e}_i\\ ... \\ \sum_{j=1}^{d}\gamma_j^{(d)}\sum_{i=1}^m u_{i,j,w}\textbf{e}_i\end{array}\right) = \left(\begin{array}{c}\sum_{i,j}u_{i,j,w}\textbf{e}_i\gamma_j^{(1)}\\ \sum_{i,j}u_{i,j,w}\textbf{e}_i\gamma_j^{(2)}\\ ... \\ \sum_{i,j}u_{i,j,w}\textbf{e}_i\gamma_j^{(d)}\end{array}\right)\\ &= \left(\begin{array}{c}\textbf{u}_w^{(1)}\\ \textbf{u}_w^{(2)}\\...\\ \textbf{u}_w^{(d)}\end{array}\right)
%\\ &
= {^t}(\textbf{v}_1^{(1)},...,\textbf{v}_w^{(d)}).\end{align*}
\end{proof}

We now wish to encode the conjugates $M^{(i)}$ and the $\tau^{(i)}$ into an $mnd\times mnd$ matrix. In order to do this we construct the block matrix $Z_0(\tau) = (Z_{i,j}(\tau))\in\C^{mnd,mnd}$, where for each $1\leq i,j\leq n$ we have $Z_{i,j}(\tau) = \text{diag}(\tau_{i,j}^{(1)}M^{(1)},...,\tau_{i,j}^{(d)}M^{(d)})\in \C^{md,md}$.

Letting $Z_1(\tau) = {^t}G'Z_0(\tau)G'$ (which is in $\h_{(mnd)}$) we can now prove the relation mentioned earlier.

\begin{lem}
For any $U\in V$ we have $M_{V,\Q}[U]\tau = {^t}\phi'(U) Z_1(\tau)\phi'(U)$.
\end{lem}

\begin{proof}
It is clear by the lemma that \begin{align*}{^t}\phi'(U) Z_1(\tau) \phi'(U) &= {^t}(G'\phi'(U)) Z_0(\tau) (G'\phi'(U))\\ &= \sum_{i,j}(\textbf{v}_i^{(1)}, \textbf{v}_i^{(2)}, ... , \textbf{v}_i^{(d)}) Z_{i,j}(\tau) {^T}(\textbf{v}_j^{(1)}, \textbf{v}_j^{(2)}, ... , \textbf{v}_j^{(d)})\\ &= \sum_{i,j} (\textbf{v}_i^{(1)}, \textbf{v}_i^{(2)},...,\textbf{v}_i^{(d)})\left(\begin{array}{c}\tau_{i,j}^{(1)}M^{(1)}\textbf{u}_j^{(1)}\\ \tau_{i,j}^{(2)}M^{(2)}\textbf{u}_j^{(2)}\\ ... \\ \tau_{i,j}^{(d)}M^{(d)}\textbf{u}_j^{(d)}\end{array}\right)\\ &= \sum_{w=1}^d\left(\sum_{i,j}\textbf{v}_i^{(w)}M^{(w)}\textbf{u}_j^{(w)}\tau_{i,j}^{(w)}\right)\\ &=\sum_{w=1}^d\left(\sum_{i,j}M[U]_{i,j}^{(w)}\tau_{i,j}^{(w)}\right)\\ &= \sum_{w=1}^d(\sigma(M[U]^{(w)}\tau^{(w)}))\\ &= \sigma\left(\sum_{w=1}^d M[U]^{(w)}\tau^{(w)}\right)\\ &= M_V(U\tau). \end{align*} In the third to last equality we use the identity $\sigma({^t}AB) = \sum_{i,j}A_{i,j}B_{i,j}$ with the symmetry of $M[U]^{(w)}$.
\end{proof}

One also proves in a similar fashion that $B_{V,\Q}(U,W) = {^t}\phi'(U)Z_1(I)\phi'(W)$, a fact we will need later. 

Let $A\in \text{GL}_{mnd}(\Q)$ be the matrix whose $i$th column is $\phi'(U_i)$. If $U = \sum_{r=1}^{mnd}u_r U_r\in\calL'$ then it is clear that $\phi'(U) = A\textbf{u}$, where $\textbf{u} = {^t}(u_1,...,u_{mnd})\in \Z^{mnd}$. Thus by the lemma we see that $M_{\calL',\Q}[U]\tau = {^t}\textbf{u} Z(\tau) \textbf{u}$, where $Z(\tau) = {^t}A Z_1(\tau) A$. In particular the map $\phi'':V \rightarrow \Q^{mnd,1}$ given by $\phi''(U) = A^{-1}\phi'(U)$ gives an isometry between $M_{\calL',\Q}$ and the quadratic form on $\Q^{mnd,1}$ with Gram matrix $\Phi(\calL') = Z(I)$.

The following properties of $Z(\tau)$ will be useful.

\begin{lem}
\begin{itemize}
\item{$-Z(\tau)^{-1} = \Phi(\calL')^{-1}Z(\tau^{-1})\Phi(\calL')^{-1}$.}
\item{$\mathrm{det}(-iZ(\tau))^{-\frac{1}{2}} = \frac{1}{\sqrt{\mathrm{det}(Z(I))}}\mathrm{det}(N(-i\tau))^{-\frac{m}{2}}$, where $N(\tau) = \prod_{w=1}^d\tau^{(w)}$.}
\end{itemize}
\end{lem}

Notice that the compliment of $\calL'$ is defined to be  the largest $\Ok$-lattice satisfying $\Tr(B_{V,\kfld}(U,W))\in\Z$ for all $U\in \calL'$ and hence is the dual of $\calL'$ with respect to $M_{V,\Q}$. Thus the Gram matrix of $\widetilde{\calL'}$ with respect to the above basis is $\Phi(\calL')^{-1} = (B_{V,\Q}(U_i, U_j))_{i,j})^{-1}$. 

Given $W\in V$ we define the shifted theta series: \[\theta(\calL',W;\tau) = \sum_{U\in\calL'}e\{Q[U+W]\tau\}.\] We are now able to prove the following inversion formula for $\theta(\calL',W;\tau)$.

\begin{thm}(Inversion formula)
For $m,n\in\Z_+$,
$V=\kfld^{m,n}$, $\calL'$ a lattice on $V$ with quadratic form given by the totally positive matrix
$M\in\kfld^{n,n}_{\sym}$ and
$W\in V$, we have \[\theta(\calL',W;\tau) = \frac{1}{\sqrt{\mathrm{det}(\Phi(\calL'))}}\mathrm{det}(N(-i\tau))^{-\frac{m}{2}}\sum_{Y\in\widetilde{\calL'}}\e\{2\,^tWMY - M[Y]\tau^{-1}\}.\]
\end{thm}

\begin{proof}
First suppose that $M$ is even integral and $W\in\widetilde{\calL'}$.
Let $\phi''(W) = \textbf{w}_0$. Then using the above lemma and discussion: 
\begin{align*}\theta(\calL',W;\tau) 
&= \sum_{U\in\calL'}e^{\pi i M_{V,\Q}[U+W]\tau}\\
 &= \sum_{\textbf{u}\in\Z^{mnd}}e^{\pi i ({^t}(\textbf{u}+\textbf{w}_0)Z(\tau)(\textbf{u}+\textbf{w}_0))}.
\end{align*} 
Since $\textbf{w}_0\in\Q^{mnd,1}$ is fixed and $Z(\tau)\in\h_{(mnd)}$, the right hand side is a generalised theta series with variable $Z(\tau)$. The inversion formula for such forms is known and applying this gives: 

\begin{align*}\theta(\calL',W;\tau) 
&= \text{det}(-iZ(\tau))^{-\frac{1}{2}}\sum_{\textbf{u}\in\Z^{mnd}} e^{\pi i(-{^t}\textbf{u} Z(\tau)^{-1}\textbf{u} - 2^{t}\textbf{u}\textbf{w}_0)} \\ 
&= \text{det}(-iZ(\tau))^{-\frac{1}{2}}\sum_{\textbf{u}\in\Z^{mnd}} e^{\pi i({^t}(Z(I)^{-1}\textbf{u}) Z(\tau^{-1})(Z(I)^{-1}\textbf{u}) - 2{^t}\textbf{u}\textbf{w}_0)}\\ 
&= \text{det}(-iZ(\tau))^{-\frac{1}{2}}\sum_{\textbf{w}\in Z(I)^{-1}\Z^{mnd}}e^{\pi i({^t}\textbf{w} Z(\tau^{-1})\textbf{w} - 2{^t}\textbf{w}Z(I)\textbf{w}_0)}\\ 
&= \text{det}(Z(-i\tau))^{-\frac{1}{2}}\sum_{\textbf{w}\in Z(I)^{-1}\Z^{mnd}}e^{\pi i(2{^t}\textbf{w}_0 Z(I)\textbf{w} - {^t}\textbf{w} Z(\tau^{-1})\textbf{w})} \\ 
&= \frac{1}{\sqrt{\text{det}(\Phi(\calL'))}}\text{det}(N(-i\tau))^{-\frac{m}{2}}\sum_{Y\in \widetilde{\calL'}}e^{\pi i(2B_{V,\Q}(W,Y) - M_{V,\Q}[Y]\tau^{-1})}\\ 
&= \frac{1}{\sqrt{\mathrm{det}(\Phi(\calL'))}}\mathrm{det}(N(-i\tau))^{-\frac{m}{2}}\sum_{Y\in\widetilde{\calL'}}\e\{2M[W,Y] - M[Y]\tau^{-1}\}. 
\end{align*}

Now suppose $M$ is not necessarily even integral and $W$ is not necessarily in $\widetilde{\calL'}$.
Take $c\in\Ok$ to be totally positive so that $cM$ is even integral and $cW\in\widetilde{\calL'}$.
Let $\calL''$ denote the lattice $\calL'$ equipped with the scaled quadratic form $cM$.  Then $\widetilde{\calL''}=c^{-1}\widetilde{\calL'}$, and hence from above we have
\begin{align*}
&\theta(\calL',W;\tau)\\
&\quad=\theta(\calL'',W;\tau/c)\\
&\quad=\frac{1}{\sqrt{\text{det}(\Phi(\calL''))}}\text{det}(N(-i\tau/c))^{-\frac{m}{2}}
\sum_{Y\in\widetilde{\calL'}}\e\{2cM[W,c^{-1}Y]-cM[c^{-1}Y]c\tau^{-1}\}\\
&\quad=\frac{1}{\sqrt{\text{det}(\Phi(\calL'))}}\text{det}(N(-i\tau))^{-\frac{m}{2}}
\sum_{Y\in\widetilde{\calL'}}\e\{2M[W,Y]-M[Y]\tau^{-1}\}.
\end{align*}

\end{proof} 

\end{subsection}

\begin{subsection}{The transformation formula.}
From hereon we assume $m$ is even with $m=2k$, and we focus on the lattice
$$\calL'=\calL\big<\I_1,\ldots,\I_n\big>=\big<\calA_1,\ldots,\calA_m\big>
\Ok^{m,n}\big<\I_1,\ldots,\I_n\big>$$
where $\I_1,\ldots,\I_n$ are (nonzero) fractional ideals fixed throughout this section,
and $\calL'$ is equipped with the quadratic form $Q$ as fixed in Section 1.
Using the inversion formula, we prove the transformation formula and thereby show that $\theta(\calL';\tau)\in\M_k(\Gamma_0(\stufe;\I_1,\ldots,\I_n;\n),\chi_L).$
For this, we note that 
$$\widetilde{\calL'}=\widetilde\calL\big<\I_1^{-1},\ldots,\I_n^{-1}\big>
=\partial^{-1}Q^{-1}\big<\calA_1^{-1},\ldots,\calA_m^{-1}\big>\Ok^{m,n}
\big<\I_1^{-1},\ldots,\I_n^{-1}\big>.$$

% We will prove this in a series of steps.

\begin{prop}  Let $\calL'$ be as above, and
suppose
$$\gamma=\begin{pmatrix}A&B\\C&D\end{pmatrix}\in\Gamma(\stufe ;\I_1,\ldots,\I_n;\n)$$
with $\det D\not=0$.
Then
\begin{align*}
&\theta(\calL';(A\tau+B)(C\tau+D)^{-1})\\
&\quad=
\det(N(-i\tau(C\tau+D)^{-1}D))^{-k} \det(N(-i\tau))^{k}\\
&\qquad \cdot
\left(\sum_{U\in\calL'/\calL'\,^tD} \e\{Q[U]BD^{-1}\}\right)\,\theta(\calL';\tau).
\end{align*}
\end{prop}

\begin{proof}  
As one can check, we have
$$(A\tau+B)(C\tau+D)^{-1}
=\,^tD^{-1}\,^tB+\,^tD^{-1}\tau(C\tau+D)^{-1}.$$
For $U,Y\in\calL'$, using that $^tD^{-1}B=\,^tBD^{-1}$, $\sigma(MN)=\sigma(NM)$
and $\sigma(M)=\sigma(\,^tM)$ for $M,N\in\kfld^{n,n}$,
we find that
$$\sigma(Q[U+Y\,^tD]\,^tD^{-1}B)
% = \sigma(Q[U]\,^tD^{-1}B)+\sigma(2\,^tUQYB)+\sigma(DQ[Y]B)
\in\sigma(Q[U]\,^tD^{-1}B+2\partial^{-1}),$$
and so $\e\{Q[U+\,^tDY]\,^tD^{-1}\,^tB\}=\e\{Q[U]\,^tD^{-1}\,^tB\}$.
Also, since 
$$D\in\big<\I_1,\ldots,\I_n\big>\Ok^{n,n}\big<\I_1^{-1},\ldots,\I_n^{-1}\big>,$$
we know that $\det D\in\Ok$ and that $\calL'\,^tD\subseteq\calL'.$
Hence using the Inversion Formula and that $^tBD$ is symmetric, we have
\begin{align*}
&\theta(\calL';(A\tau+B)(C\tau+D)^{-1})\\
&\quad =
\sum_{U\in\calL'/\calL'\,^tD} \e\{Q[U]\,^tD^{-1}\,^tB\}
\theta(\calL',U\,^tD^{-1};\tau(C\tau+D)^{-1}D)\\
&\quad=
\frac{1}{\sqrt{\Phi(\calL')}}\ \det(N(-i\tau(C\tau+D)^{-1}D))^{-k}\\
&\qquad\cdot
\sum_{U\in\calL'/\calL'\,^tD} \e\{Q[U]BD^{-1}\}
\sum_{Y\in\widetilde{\calL'}} 
\e\{2\,^tYQU\,^tD^{-1}-Q[Y](D^{-1}C+\tau^{-1})\}\\
&\quad=
\frac{1}{\sqrt{\Phi(\calL)}}\ \det(N(-i\tau(C\tau+D)^{-1}D))^{-k}
\sum_{Y\in\widetilde{\calL'}}\e\{-Q[Y]\tau^{-1}\}\\
&\qquad\cdot
\left(\sum_{U\in\calL'/\calL'\,^tD} \e\{Q[U]BD^{-1}
+2\,^tYQU\,^tD^{-1} - Q[Y]D^{-1}C\}\right).
\end{align*}
We claim that this last sum on $U$ is independent of $Y$.  First,
using that $D\,^tA-C\,^tB=I=\,^tDA-\,^tBC$) we have that
$$\e\{-Q[UB+Y]D^{-1}C\}=\e\{Q[U]BD^{-1}+2\,^tYQU\,^tD^{-1}-Q[Y]D^{-1}C\}.$$
We know that $\big<\calA_1,\ldots,\calA_m\big>Q\big<\calA_1,\ldots,\calA_m\big>\subseteq
\n\Ok^{m,m},$ so for $U\in\calL'$, we have $UB\in\widetilde{\calL'};$
we also have ${\calL'}\,^tD\subseteq{\calL'}$.
For fixed $Y\in\widetilde{\calL'}$,
we show that $UB+Y$ varies over 
$\widetilde\calL'/\widetilde\calL'D$ as $U$ varies over $\calL'/\calL'\,^tD$;
to do this, we show that 
 for $U\in\calL'$, we have $UB\in\widetilde\calL'D$ if and only if $U\in\calL'\,^tD$.

We will argue locally.  We first observe that for $\calP$ a prime with $\calP\nmid\det D$, we have
$$\Ok_{\calP}\calL'\,^tD=\Ok_{\calP}\calL',\ 
\Ok_{\calP}\widetilde{\calL'}D=\Ok_{\calP}\widetilde{\calL'}.$$
Thus we need to show that for any prime $\calP|\det D$, we have $U\in\Ok_{\calP}\calL'\,^tD$ if and only if $UB\in\widetilde{\calL'}D$.  We first fix some notation.

Choose $\beta, \eta, \alpha_i, \mu_{\ell}\in\kfld$ so that for every prime $\calP|\det D$, we have
$$\beta\Ok_{\calP}=\partial\Ok_{\calP},\ \eta\Ok_{\calP}=\n\Ok_{\calP},\ 
\alpha_i\Ok_{\calP}=\calA_i\Ok_{\calP},\ \mu_{\ell}\Ok_{\calP}=\I_{\ell}\Ok_{\calP}$$ 
$(1\le i\le m,\ 1\le\ell\le n).$
Set $\underline\alpha=(\alpha_1,\ldots,\alpha_m),$
$\underline\mu=(\mu_1,\ldots,\mu_n),$
$$\begin{pmatrix}A'&B'\\C'&D'\end{pmatrix}=
\begin{pmatrix}\beta\eta\underline\mu\\&\underline\mu^{-1}\end{pmatrix}
\begin{pmatrix}A&B\\C&D\end{pmatrix}
\begin{pmatrix}\beta^{-1}\eta^{-1}\underline\alpha^{-1}\\
&\underline\mu^{-1}\end{pmatrix},$$
$U'=\underline\alpha^{-1}U\underline\mu^{-1}$, 
$Q'=\eta^{-1}\underline\alpha Q\underline\alpha.$

Now fix a prime $\calP|\det D$.  We have 
$\begin{pmatrix}A'&B'\\C'&D'\end{pmatrix}\in Sp_n(\Ok_{\calP})$ with $\stufe|C'$,
$U'\in\Ok_{\calP}^{m,n}$, and $Q'\in\Ok_{\calP}^{m,m}.$
Let $\calL'_{\calP}=\Ok_{\calP}\calL'$, 
$\widetilde{\calL'_{\calP}}=\Ok_{\calP}\widetilde{\calL'}$;
so $\calL'_{\calP}=\underline\alpha\Ok_{\calP}^{m,n}\underline\mu$,
$\widetilde{\calL'_{\calP}}
=\beta^{-1}Q^{-1}\underline\alpha^{-1}\Ok_{\calP}^{m,n}\underline\mu^{-1}.$
Also, $U\in\calL'_{\calP}\,^tD$ if and only if $U'\in\Ok_{\calP}^{m,n}\,^tD'$, and
$UB\in\widetilde{\calL'_{\calP}}D$ if and only if
$Q'U'B'\in\Ok_{\calP}^{m,n}D'$.  Since $\calP|\det D$ and hence $\calP\nmid\stufe$,
we have that $Q'$ is unimodular over $\Ok_{\calP}$ (see the discussion in Section 1); thus
$UB\in\widetilde{\calL'_{\calP}}D$ if and only if $U'B'\in\Ok_{\calP}^{m,n}D'$.
Choose $E,G\in GL_n(\Ok_{\calP})$ so that
$$E\,^tD'G=\begin{pmatrix}I_r&0\\0&\pi D_1\end{pmatrix}$$
where $r=\rank_{\calP}D'$ and $\pi\Ok_{\calP}=\calP\Ok_{\calP}$
(here $\rank_{\calP}D'$ means the rank of $D'$ as a matrix over
$\Ok_{\calP}/\calP\Ok_{\calP}$).
Write 
$$E\,^tB'\,^tG^{-1}=\begin{pmatrix}B_{00}&B_{01}\\B_{10}'&B_{11}\end{pmatrix}$$
where $B_{00}$ is $r\times r$.  By the symmetry of $^tB'D'$, 
we get $B_{01}\equiv0\ (\calP\Ok_{\calP}),$ and by the fact that 
$\begin{pmatrix}A'&B'\\C'&D'\end{pmatrix}\in SL_{2n}(\Ok_{\calP})$, we find that
$B_{11}$ is invertible over $\Ok_{\calP}$.
Now take 
$$X=G^{-1}B'\,^tE+\begin{pmatrix}I-\,^tB_{00}\\&0\end{pmatrix}\,^tGD'\,^tE
=\begin{pmatrix}I&\pi B_{01}\,^tD_1\\^tB_{01}&^tB_{11}\end{pmatrix};$$
so $X$ is an invertible matrix in $\Ok_{\calP}^{n,n}$.
Also, an easy check shows that
$\,^tX(\,^tGD'\,^tE)=(E\,^tD'G)X$.  Then
\begin{align*}
U'B'\in\Ok_{\calP}^{m,n}D'
%&\iff
%(U'G)(G^{-1}B'\,^tE)\in \Ok_{\calP}^{m,n}(\,^tGD'\,^tE)\\
%&\iff
% (U'G)(G^{-1}B'\,^tE)+(U'G)(\,^tGD'\,^tE)\in \Ok_{\calP}^{m,n}(\,^tGD'\,^tE)\\
 &\iff
U'GX \in \Ok_{\calP}^{m,n}(\,^tGD'\,^tE)\\
%&\iff
%U'G \in \Ok_{\calP}^{m,n}(\,^tGD'\,^tE)X^{-1}\\
&\iff
U'G\in\Ok_{\calP}^{m,n}\,^tX^{-1}(E\,^tD'G)\\
&\iff
U'\in\Ok_{\calP}^{m,n}\,^tX^{-1}E\,^tD'=\Ok_{\calP}^{m,n}\,^tD'.
\end{align*}
Since this holds for all prime ideals $\calP|\det D$, we see that as $U$ varies over $\calL'/\calL'\,^tD$, $UB+Y$ varies over $\widetilde{\calL}'/\widetilde{\calL}' D$.

Thus in our last expression for $\theta(\calL';(A\tau+B)(C\tau+D)^{-1})$, we can simplify the sum on $U$ (eliminating the terms with $Y$), and reverse the order of summation.  Then another application of the Inversion Formula (Theorem 2.4) yields the proposition.
\end{proof}

Now we evaluate the sum on $U$ in the above propostion.

\begin{prop}  With the notation be as in the previous proposition, we have
$$\sum_{U\in\calL'/\calL'\,^tD}\e\{Q[U]BD^{-1}\}
= N(\det D)^k\,\chi_L(\det D)$$
where $\chi_L$ is as defined in the introduction.
\end{prop}

\begin{proof}  Let the notation be as in the previous proof.
With $\frakD=(\det D)\Ok$, we have
\begin{align*}
&\sum_{U\in\calL'/\calL'\,^tD} \e\{Q[U]BD^{-1}\}\\
&\quad=
\sum_{U\in\Ok^{m,n}/\Ok^{m,n}\,^tD}
\e\left\{Q\left[
\big<\alpha_1,\cdots,\alpha_m\big>U
\big<\mu_1,\cdots,\mu_n\big>\right]BD^{-1}\right\}\\
&\quad=
\sum_{U\in\Ok^{m,n}/\Ok^{m,n}\,^tD'}
\e\{\beta' Q'[U]B'(D')^{-1}\}\\
&\quad=
N(\frakD)^{m(1-n)}
\sum_{U\in\Ok^{m,n}/\frakD\Ok^{m,n}} \e\{\beta' Q'[U]B'(D')^{-1}\}\\
&\quad=
\prod_{\calP^e\parallel\frakD} N(\calP^e)^{m(1-n)}
\sum_{U\in\calP^{-e}\frakD\Ok^{m,n}/\frakD\Ok^{m,n}}
\e\{\beta' Q'[U]B'(D')^{-1}\}.
\end{align*}
Now rather standard techniques (such as those used in Section 5 \cite{theta eis}, which are local arguments) can be used to
reduce this computation to computations of more manageable sums.  

For $\calP\nmid 2$, we are left with computing sums of the shape
$$\sum_{u\in\Ok^{m,1}/\calP\Ok^{m,1}}\e\{Q[u]\omega\}$$
where $\omega\in\partial^{-1}\calP^{-1}$ with $\omega\Ok_{\calP}=\partial^{-1}\calP^{-1}\Ok_{\calP}$.
We replace $u$ by $Eu$ where $E\in SL_n(\Ok)$ so that
$Q[E]\equiv 2Q'\ (\calP)$ where $Q'\in\Ok^{m,m}$ is diagonal modulo $\calP$; this reduces the computation to that of
evaluating $\sum_{x\in\Ok/\calP}\e\{2x^2b\omega\}$ where $b\in\Ok\smallsetminus\calP$, and now standard techniques can be used
to show that
$$\sum_{x\in\Ok/\calP}\e\{2x^2b\omega\}=\left(\frac{b}{\calP}\right)
\sum_{x\in\Ok/\calP}\e\{2x^2\omega\}$$
and
$$\left(\sum_{x\in\Ok/\calP}\e\{2x^2\omega\}\right)^2=\left(\frac{-1}{\calP}\right)N(\calP).$$

Suppose $\calP|2$; then $u$ can be replaced by $Eu$ where $E\in SL_n(\Ok)$ so that modulo $2\calP$, $Q[E]$ is either an orthonal sum of $k$ copies of the matrix $\begin{pmatrix}0&1\\1&0\end{pmatrix}$, or it is the orthogonal sum of $k-1$ copies of $\begin{pmatrix}0&1\\1&0\end{pmatrix}$ and 1 copy of $\begin{pmatrix}2&1\\1&2\varepsilon\end{pmatrix}$ where $\varepsilon\in\Ok$ so that in $\Ok_{\calP}$, $1-4\varepsilon$ has quadratic defect $4\Ok_{\calP}$.  Thus the sums to be evaluated are now
$$\sum_{x,y\in\Ok/\calP}\e\{2xy\omega\},\ 
\sum_{x,y\in\Ok/\calP}\e\{2(x^2+xy+\varepsilon y^2\omega)\}$$
where $\omega\in\partial^{-1}\calP^{-1}$ with $\omega\Ok_{\calP}=\partial^{-1}\calP^{-1}\Ok_{\calP}$.
To evaluate the first sum, we sum first on $x$; when $x\not\in\calP$, this is a complete character sum yielding 0, and when $x\in\calP$ we get $N(\calP)$.
Now consider the second sum; when $x\not\in\calP$, we replace $y$ by $xy$ and recall that as $x$ varies over $(\Ok/\calP)^{\times}$, so does $x^2$.  Hence we get
\begin{align*}
&\sum_{x,y\in\Ok/\calP}\e\{2(x^2+xy+\varepsilon y^2)\omega\}\\
&\quad=
\sum_{x\in(\Ok/\calP)^{\times}}\sum_{y\in\Ok/\calP}\e\{2x^2(1+y+\varepsilon y^2)\omega\}
%\\&\qquad 
+ \sum_{y\in\Ok/\calP}\e\{2\varepsilon y^2\omega\}\\
&\quad=
\sum_{y\in\Ok/\calP}\sum_{x\in(\Ok/\calP)^{\times}}\e\{2x(1+y+\varepsilon y^2)\omega\}
%\\&\qquad 
+ \sum_{y\in\Ok/\calP}\e\{2\varepsilon y\omega\}.
\end{align*}
We have $\sum_{y\in\Ok/\calP}\e\{2\varepsilon y\omega\}=0$, as this sum on $y$ is a complete character sum (with a nontrivial character).  Similarly,
$$\sum_{x\in\Ok/\calP}\e\{2x(1+y+\varepsilon y^2)\omega\}=0$$
whenever $1+y+\varepsilon y^2\not\in\calP$.  Since $1-4\varepsilon$ has quadratic defect $4\Ok_{\calP}$ in $\Ok_{\calP}$, one checks that we cannot have $1+y+\varepsilon y^2\in\calP$ for any $y\in\Ok$.  Thus
$$\sum_{x,y\in\Ok/\calP}\e\{2(x^2+xy+\varepsilon y^2)\omega\}
=-N(\calP).$$

The proposition now follows from the definition of $\chi_L$, as described in the introduction.
\end{proof}

Now we can state the main result of this section.

\begin{thm} (Transformation Formula)
 Let $\calL'=\calL\big<\I_1,\ldots,\I_n\big>$, $\gamma=\begin{pmatrix}A&B\\C&D\end{pmatrix}
\in\Gamma(\stufe;\I_1,\ldots,\I_n;\n),$ and let
$\chi_L$ 
be as defined in the introduction.  
We have
$$\theta(\calL';\tau)|\gamma=\chi_L(\det D)\,\theta(\calL';\tau).$$
%where
%$\chi_L(\det D)=N(\sgn(\det D))^{k}\chi_L^*((\det D)\Ok).$
\end{thm}

\begin{proof}  
When $\det D\not=0$, this is proved by the previous two propositions.  So suppose that $\det D=0$.
This can only happen in the case that $\stufe=\Ok$, in which case $\chi_L$ is the trivial character modulo $\Ok$.  Using that the rank of $(C\ D)$ is $n$ and $C\,^tD$ is symmetric, one can find $G\in GL_n(\Ok)$ and $W\in\Ok^{n,n}$ so that with
$$\begin{pmatrix}A'&B'\\C'&D'\end{pmatrix}=
\begin{pmatrix}A&B\\C&D\end{pmatrix}
\begin{pmatrix}^tG^{-1}&W\\0&G\end{pmatrix},$$
$\det D'\not=0$.  Thus
$$\theta(\calL';\tau)=
\theta(\calL';\tau)|\begin{pmatrix}A'&B'\\C'&D'\end{pmatrix},$$
and so
$$\theta(\calL';\tau)=\theta(\calL';\tau)|
\begin{pmatrix}^tG^{-1}&W\\0&G\end{pmatrix}^{-1}=
\theta(\calL';\tau)|\begin{pmatrix}A'&B'\\C'&D'\end{pmatrix}.$$
Thus the theorem follows.
\end{proof}

\end{subsection}
\bigskip

\section{Action of the $S_{\ell}(\calQ)$ operators on theta series}
\smallskip

In \cite{CW}, when  $\stufe =\Ok$ and $\chi_L=1$, we defined linear maps whose composition takes
$$\M^{(n)}_k(\Gamma_0(\Ok;\calQ_1^{-1}\I_1,\cdots,\calQ_n^{-1}\I_n;\n )) \text{ to }
\M{(n)}_k(\Gamma_0(\Ok;\I_1,\ldots,\I_n;\n )).$$  Here we generalize these maps to allow nontrivial level and character.  Then we evaluate the action of the maps on theta series.  

%In the following section, we will partially realize the action of the Hecke operators on theta series in terms of these maps.

Fix $\ell$ where $1\le \ell\le n$, and fix a fractional ideal $\calQ$ so that for every prime ideal $\calP|\stufe $, we have $\ord_{\calP}(\calQ)=0$.  With $\I_1,\ldots,\I_n$ fixed fractional ideals,
set
$$\Gamma'=\Gamma_0(\stufe ;\I_1,\ldots,\I_n;\n ) \text{ and }\Gamma''=\Gamma_0(\stufe ;\I_1',\ldots,\I_n';\n )$$
where
$$\I_i'=\begin{cases}\I_i&\text{if $i\not=\ell$,}\\ \calQ^{-1}\I_{\ell}&\text{otherwise.}
\end{cases}$$
Take 
$$\begin{pmatrix}w&x\\y&z\end{pmatrix}\in
\begin{pmatrix}\calQ&\calQ(\I_{\ell}^2\n \partial)^{-1}\\
\calQ^{-1}\stufe \I_{\ell}^2\n \partial&\calQ^{-1}\end{pmatrix}$$
so that $wz-xy=1$.  Set
\begin{align*}
W&=\big<I_{\ell-1},w,I_{n-\ell}\big>,\ 
X=\big<0_{\ell-1},x,0_{n-\ell}\big>,\\
Y&=\big<0_{\ell-1},y,0_{n-\ell}\big>,\ 
Z=\big<I_{\ell-1},z,I_{n-\ell}\big>,\\
\delta&=\begin{pmatrix}W&X\\Y&Z\end{pmatrix}.
\end{align*}
A straightforward check shows that $\delta\Gamma'\delta^{-1}\subseteq\Gamma''$ and $\delta^{-1}\Gamma''\delta\subseteq\Gamma'$.

We will use $\delta$ to define $S_{\ell}(\calQ):\M^{(n)}_k(\Gamma'',\chi)\to\M^{(n)}_k(\Gamma',\chi)$.
Toward this, we prove following.

\begin{prop}  Let $L=\calA_1x_1\oplus\cdots\oplus\calA_mx_m$,
$\calL=\big<\calA_1,\ldots,\calA_m\big>\Ok^{m,n}$ be as fixed in Section 1, with $\q$ the quadratic form on $L$ given by the matrix $Q$ (relative to the basis
$(x_1,\ldots,x_m)$).
With $\I_j$, $\I'_j$ ($1\le j\le n$) and $\delta$ as above, set
$$\calL'=\calL\big<\I_1,\cdots,\I_n\big>
\text{ and }
\calL''=\calL\big<\I'_1,\cdots,\I'_n\big>.$$
Then
\begin{align*}
\theta(\calL'';\delta\tau)
&=
% \frac{\sqrt{\Phi(\calL')}}{\sqrt{\Phi(\calL'')}} NO
\det(N(-i\tau(Y\tau+Z)^{-1}Z))^{-m/2} \det(N(-i\tau))^{m/2}\\
&\quad\cdot
\sum_{U\in\calL''/\calL'Z}\e\{Q[U]XZ^{-1}\}\theta(\calL';\tau).
\end{align*}
\end{prop}

\begin{proof}  We have
\begin{align*}
\theta(\calL'';\delta\tau)
=\sum_{U\in\calL''/\calL'\,Z}\e\{Q[U]XZ^{-1}\}
\theta(\calL',UZ^{-1};\tau(Y\tau+Z)^{-1}Z).
\end{align*}
Then we follow the argument of Proposition 2.5 to finish proving this proposition.
(To see that with $U'\in\widetilde{\calL'}$, $UX+U'$ varies over $\widetilde\calL'/\widetilde\calL''Z$ as $U$ varies over $\calL''/\calL' Z$:  take a prime ideal $\calP$ and set $e=\ord_{\calP}\calQ$, and take $\pi\in\calP$ so that $\Ok_{\calP}\pi=\Ok_{\calP}\calP$.  Take $z'\in\Ok_{\calP}$ so that $z=\pi^ez'$.  Then with
$Z'=\big<I_{\ell-1},z',I_{n-\ell}\big>$, we have natural isomorphisms
$\Ok_{\calP}\calL''/\Ok_{\calP}\calL'Z \approx \Ok_{\calP}\pi^{-e}\I_{\ell}L/\Ok_{\calP}\pi^{-e}\I_{\ell}Lz'$ and
$\Ok_{\calP}\widetilde\calL'/\Ok_{\calP}\widetilde\calL''Z\approx \Ok_{\calP}\I_{\ell}^{-1}\widetilde L /\Ok_{\calP}\I_{\ell}^{-1}\widetilde Lz'.$
Then the argument used in Proposition 2.1 \cite{thesis} show that $UX+U'$ varies over $\widetilde \calL'/\widetilde\calL''Z$ as $U$ varies over $\calL''/\calL'Z$.
So the sum on $U$ is independent of $U'$ and hence can be taken with $U'=0$.)
\end{proof}

\smallskip\noindent
{\bf Definition.}  Let $\calQ$, $\delta$, $\Gamma'$, and $\Gamma''$ be as at the beginning of this section.  Take $f\in\M{(n)}_k(\Gamma'',\chi)$.  We define
$$f|S_{\ell}(\calQ)
=\overline\chi(\delta)\,f|\delta.$$
One easily verifies that if one changes the choice of $\delta$ (subject to the conditions placed on this choice), $S_{\ell}(\calQ)$ is well defined.  Also,
for $\begin{pmatrix}A&B\\C&D\end{pmatrix}\in\Gamma'$ and
$\begin{pmatrix}A'&B'\\C'&D'\end{pmatrix}
=\delta\begin{pmatrix}A&B\\C&D\end{pmatrix}\delta^{-1},$
we have $\det D'\equiv \det D\ (\stufe ).$
Thus 
$$S_{\ell}(\calQ):\M{(n)}_k(\Gamma'',\chi)\to\M{(n)}_k(\Gamma',\chi).$$
Also note that for $1\le \ell'\le n$ and $\calQ'$ a fractional ideal with order 0 at any prime ideal dividing $\stufe $, we have $S_{\ell}(\calQ)S_{\ell'}(\calQ')=S_{\ell'}(\calQ')S_{\ell}(\calQ).$
\smallskip

Now the techniques used to prove Propositions 2.5 and 2.6 give us the following.

\begin{thm}  Let $\calQ$ be a fractional ideal so that for every prime ideal $\calP|\stufe$, we have $\ord_{\calP}\calQ=0$.  With $\calL', \calL''$ as in Proposition 3.1 and $S(\calQ)$ defined as above, we have
$$\theta(\calL'';\tau)|S_{\ell}(\calQ)
=N(\calQ)^k \chi_L^*(\calQ)\,\theta(\calL';\tau).$$
\end{thm}

\bigskip
\section{Action of the Hecke operators $T_j(\calP^2)$ on theta series}
\smallskip

Recall that we have fixed the lattice $L$.
For the duration of this section,
 we fix an integer $j$ with $1\le j\le n$, and we fix a prime ideal $\calP$
with $\calP\nmid\stufe$.
We evaluate 
$$\theta(L;\tau)|S_1(\calP^{-1})\cdots S_j(\calP^{-1}) T_j(\calP^2),$$
which lies in the same space as $\theta(L;\tau)$.

\smallskip\noindent
{\bf Local notation.}
Throughout this section, we take $\beta\in\partial$, $\beta'\in\partial^{-1}$ so that $\beta\beta'\equiv1\ (\calP)$, $\eta\in\n$, $\eta'\in\n$ so that $\eta\eta'\equiv1\ (\calP)$, $\pi\in\calP$, $\pi'\in\calP^{-1}$ so that $\pi\pi'\equiv1\ (\calP)$.
Also, with $r_0,r_2\in\Z_{\ge0}$ so that $r_0+r_2\le j$,
$r_1=j-r_0-r_2$, we set
$$P'_{r_2}=P'_{j;r_2}=\begin{pmatrix}I_{j-r_2}\\&0&I_{n-j}\\&I_{r_2}&0\end{pmatrix},\ 
P_{r_2}=P_{j;r_2}=\begin{pmatrix}P'_{j;r_2}\\&P'_{j;r_2}\end{pmatrix},$$
$$X_{r_0,r_2}=X^{(n)}_{r_0,r_2}=\big<\calP I_{r_0},I_{n-r_0-r_2},\calP^{-1}I_{r_2}\big>,
\ X_{r_1}=X^{(r_1+n-j)}_{r_1}=\big<\calP I_{r_i},I_{n-j}\big>,$$
and 
$$\K_{r_0,r_2}=\K^{(n)}_{r_0,r_2}=X_{r_0,r_2}GL_n(\Ok)
X_{r_0,r_2}^{-1}\cap GL_n(\Ok),$$
$$\K'_{r_1}=\K^{(n-r_0-r_2)}_{r_1}=X'_{r_1}GL_{n-r_0-r_2}(\Ok)(X'_{r_1})^{-1}
\cap GL_{n-r_0-r_2}(\Ok).$$

\smallskip
Using this notation, we have the following.

\begin{prop}  
With $r_0,r_2$ non-negative integers so that $r_0+r_2\le j$, set
$$\calL_{r_0,r_2}=\calL\big<\calP^{-1} I_{r_0}, I_{n-r_0-r_2},\calP I_{r_2}\big>.$$
Then with $r_0,r_2$ varying subject to the above conditions,
we have
$$\theta(L;\tau)|S_1\cdots S_j(\calP^{-1})T_j(\calP^2)
=\sum_{r_0,r_2}\theta(L;\tau)|S_1\cdots S_j(\calP^{-1})\,T_{j;r_0,r_2}(\calP^2)$$
where
\begin{align*}
&N(\calP)^{kj}\chi^*_L(\calP)^j \theta(L;\tau)|S_1\cdots S_j(\calP^{-1})T_{j;r_0,r_2}(\calP^2)\\
&\quad
=\chi_L(\det P'_{r_2})\, \sum_{Y',G}\chi_L(\det G)\, \theta(\calL_{r_0,r_2};\tau)|
\begin{pmatrix}G^{-1}&Y'\,^tG\\&^tG\end{pmatrix}|P_{r_2};
\end{align*}
here $Y', G$ are defined as follows.
We have

$$\begin{pmatrix}G^{-1}&Y'\,^tG\\&^tG\end{pmatrix}
=\begin{pmatrix}I&W\\&I\end{pmatrix} \begin{pmatrix}I&Y\\&I\end{pmatrix}
\begin{pmatrix}(G_0G_1)^{-1}\\&^t(G_0G_1)\end{pmatrix}
$$

where
$$W=\begin{pmatrix}0_{r_0}\\&W'\\&&0_{n-j-r_0}\end{pmatrix},
\ Y=\begin{pmatrix}Y_0&Y_2&0\\^tY_2\\0\end{pmatrix}$$
with 
$W'\in(\n\partial\calP)^{-1}\Ok^{r_1,r_1}_{\sym}$ so that
$\eta\delta\pi W'$ varies over $(\Ok^{r_1,r_1}_{\sym}/\calP \Ok^{r_1,r_1}_{\sym})^{\times},$
$Y_0\in(\n\partial)^{-1}\Ok^{r_0,r_0}_{\sym}$ so that
$\eta\delta Y_0$ varies over $\Ok^{r_0,r_0}_{\sym}/\calP^2 \Ok^{r_0,r_0}_{\sym},$
and
$Y_2\in(\n\partial)^{-1}\Ok^{r_0,n-r_2}$ so that
$\eta\delta Y_2$ varies over $\Ok^{r_0,n-r_2}/\calP \Ok^{r_0,n-r_2}$;
$$G_1=\begin{pmatrix}0_{r_0}\\&G_1'\\&&0_{n-j-r_0}\end{pmatrix}$$
with $G_1', G_0$ varying subject to the conditions $G_1'\in GL_{n-r_0-r_2}/\K'_{r_1},$ and
 $G_0\in GL_n(\Ok)/\,^t\K_{r_0,r_2}$
(so $G_0^{-1}$ varies over $\K_{r_0,r_2}\backslash GL_n(\Ok)$).
\end{prop}

\begin{proof}  We know from Section 3 that $\theta(L;\tau)|S_1\cdots S_j(\calP^{-1})$
lies in the space $\M_k(\Gamma_0(\stufe;\I_1,\cdots,\I_n;\n),\chi_L)$ where
$I_{\ell}=\calP^{-1}$ for $1\le\ell\le j$, $I_{\ell}=\Ok$ otherwise.
It is simple to generalize \cite{CW} to find matrices for 
$$T_j(\calP^2): \M_k(\Gamma_0(\stufe;\I_1,\cdots,\I_n;\n),\chi_L)
\to \M_k(\Gamma_0(\stufe;\Ok,\cdots,\Ok;\n),\chi_L).$$
(To be well-defined, we precede the action of a matrix $\begin{pmatrix}A&B\\C&D\end{pmatrix}$ by
$\overline\chi_L(\det D)$. Also, to more easily describe the matrices from \cite{CW}, we use the notation from \cite{W eis series}.)
From this, and recalling that $\chi_L$ is quadratic, we have
\begin{align*}
&\theta(L;\tau)|S_1\cdots S_j(\calP^{-1})|T_j(\calP^2)\\
&\quad=
\sum_{Y',G} \chi_L(\det G)
\,\theta(L;\tau)|S_1\cdots S_j(\calP^{-1})|S_{r_0+1}\cdots S_{j-r_2}(\calP) S_{j-r_2+1}\cdots S_j(\calP^2)\\
&\qquad\qquad
|\begin{pmatrix}G^{-1}&Y'\,^tG\\&^tG\end{pmatrix}.
\end{align*}
We also know that $\theta(L;\tau)|P_{r_2}=\chi_L(\det P'_{r_2})\,\theta(L;\tau),$
and 
\begin{align*}
&\theta(L;\tau)|P_{r_2}|S_1\cdots S_{r_0}(\calP^{-1})S_{j-r_2+1}\cdots S_j(\calP)|\,^tP_{r_2}\\
&\quad
=\theta(L;\tau)|S_1\cdots S_{r_0}(\calP^{-1})S_{n-r_2+1}\cdots S_n(\calP).
\end{align*}
Also, from Theorem 3.2, we know that
\begin{align*}
&\theta(L;\tau)|S_1\cdots S_{r_0}(\calP^{-1})S_{n-r_2+1}\cdots S_n(\calP)\\
&\quad
=N(\calP)^{k(r_2-r_0)}\chi^*_L(\calP)^{r_2-r_0}\,\theta(\calL_{r_0,r_2};\tau).
\end{align*}
From this the proposition easily follows.  (Recall that $\chi_L$ is quadratic.)
\end{proof}

Our next step in analyzing $\theta(L;\tau)|S_1\cdots S_j(\calP^{-1})T_j(\calP^2)$ is to find a more convenient way to write $\theta(\calL_{r_0,r_2};\tau)$.  Toward this, we introduce some terminology.

\smallskip

\noindent{\bf Terminology.}  Fix non-negative integers $r_0,r_2$ so that $r_0+r_2\le j$.  Take
$U\in\calL_{r_0,r_2}.$  With $(y_1\, \cdots\, y_n)=(x_1\,\cdots\,x_m)U$, set
$$\Omega(U)=\Ok y_1\oplus\cdots\oplus\Ok y_n$$
(an external direct sum).  We call $\Omega(U)$ the (free) lattice associated to $U$,
and we call $(y_1\, \cdots\, y_n)$ the basis determined by $U$; note that by the definition of $\calL$ and $\calL_{r_0,r_2}$, we have 
$$y_i\in\begin{cases}\calP^{-1}L&\text{if $1\le i\le r_0$,}\\
L&\text{if $r_0<i\le n-r_2$,}\\
\calP L&\text{otherwise.}
\end{cases}
$$
 For $U,U'\in\calL_{r_0,r_2}$, we say $U$ and $U'$ are equivalent in $\calL_{r_0,r_2}$, and write $U\sim U'$,
if there is some $G\in GL_n(\Ok)$ so that $U'=UG$.  Note that when $U\sim U'$ for $U,U'\in\calL_{r_0,r_2}$,
we have $\Omega(U)=\Omega(U')$.  
\smallskip

\begin{prop}  Fix non-negative integers $r_0,r_2$ so that $r_0+r_2\le j$.
\begin{enumerate}
\item[(a)]  
Take $U\in\calL_{r_0,r_2}$.  There are invariants $d_0,d_1$ of the equivalence class of $U$ and a representative $U_{\underline y}$ for this equivalence class so that with
$$\underline y=(y_1\,\cdots\,y_n)=(x_1\,\cdots\,x_m)U_{\underline y}$$
we have
$$y_i\in\begin{cases} \calP^{-1}L\smallsetminus L&\text{if $1\le i\le d_0$,}\\
L\smallsetminus \calP L&\text{if $d_0<i\le d_0+d_1$,}\\
\calP L&\text{otherwise}
\end{cases}$$
where $y_1,\ldots,y_{d_0+d_1}$ are linearly independent in the vector space $\kfld L$.
We call such an equivalence class representative $U_{\underline y}$ a reduced representative, and we call such a basis $\underline y$ a reduced basis; to ease notation, we write $\Omega(\underline y)$ to denote $\Omega(U_{\underline y})$.
\item[(b)]  Take $U\in\calL_{r_0,r_2}$ and take $\underline y=(y_1\,\cdots\,y_n)$ to be a reduced basis for $\Omega(U)$; take $U_{\underline y}\sim U$ so that $\underline y=(x_1\,\cdots\,x_m)U_{\underline y}.$  With $d_0,d_1$ the invariants associated to the equivalence class of $U$ (as defined in (a)),
set
\begin{align*}
\Delta=\Delta(\underline y)
&=\calP y_1\oplus\cdots\oplus\calP y_{d_0}
\oplus \Ok y_{d_0+1}\oplus\cdots\oplus\Ok y_{d_0+d_1}\\
&\quad
\oplus \calP^{-1} y_{d_0+d_1+1}\oplus\cdots\oplus\calP^{-1} y_n;
\end{align*}
we call $\Delta(\underline y)$ the formal intersection of $\calP^{-1}\Omega(\underline y)$ and $L$.
Then with $\Omega=\Omega(\underline y)$, the equivalence class of $U_{\underline y}$ in $\calL_{r_0,r_2}$ is partitioned into $\K_{r_0,r_2}$-orbits of the form $U_{\underline y}G_{\underline y,\Lambda}\cdot\K_{r_0,r_2}$
where the parameter $\Lambda$ varies over all lattices so that 
$\calP\Omega\subseteq\Lambda\subseteq\Delta$ with
$\mult_{\{\Omega:\Lambda\}}(\calP)=r_0$ and 
$\mult_{\{\Omega:\Lambda\}}(\calP^{-1})=r_2$.
Further, $G_{\underline y,\Lambda}\in GL_n(\Ok)$ is chosen so tha
$U_{\underline y}G_{\underline y,\Lambda}\in\calL_{r_0,r_2}$ and with
$$(z_1\,\cdots\,z_n)=(y_1\,\cdots\,y_n)G_{\underline y,\Lambda}
=(x_1\,\cdots\,x_m)U_{\underline y}G_{\underline y,\Lambda},$$
we have
\begin{align*}
\Lambda&=\calP z_1\oplus\cdots\oplus\calP z_{r_0}
\oplus\Ok z_{r_0+1}\oplus\cdots\oplus\Ok z_{n-r_2}\\
&\qquad
\oplus\calP^{-1} z_{n-r_2+1}\oplus\cdots\oplus\calP^{-1}z_n.
\end{align*}
\item[(c)] 
Let $\underline y$ vary so that $U_{\underline y}$ varies over the equivalence class representatives in $\calL_{r_0,r_2}$, and for each $\underline y$, let $\Lambda=\Lambda(\underline y)$ vary as in (b).  Then with $G_{\underline y,\Lambda}\in GL_n(\Ok)$ as in (b) and $E$ varying over $\K_{r_0,r_2}$, we have
$$\theta(\calL_{r_0,r_2};\tau)=\sum_{\underline y, E}\sum_{\Lambda}
\e\{Q[U_{\underline y}G_{\underline y,\Lambda}E]\tau\}.$$
\end{enumerate}
\end{prop}

\begin{proof}
(a)  Fix $U\in\calL_{r_0,r_2}$;  recalling that $L=\calA_1x_1\oplus\cdots\oplus\calA_mx_m$, let
$(y_1'\,\cdots\,y_n')=(x_1\,\cdots\,x_m)U$ and let $\Omega'=\Ok y_1'+\cdots+\Ok y_n'$.
%Note that as $U\in\calL_{r_0,r_2}$, we have
%$$y_i'\in\begin{cases} \calP^{-1}L&\text{if $1\le i\le r_0$,}\\
%L&\text{if $r_0<i\le n-r_2$,}\\
%\calP L&\text{if $n-r_2<i\le n$.}
%\end{cases}$$
We construct a representative for the equivalence class of $U$ in two steps.

{\it Step 1.}  For $x\in\calP^{-1}L$, let $\overline x=x+L$ and let $\overline{\Omega'}$ denote the image of $\Omega'$ in $\calP^{-1}L/L$ (a vector space over $\Ok/\calP$).
Thus $\{\overline y_1',\ldots,\overline y_{r_0}'\}$ spans $\overline{\Omega'}$.
Let $d_0=\dim\overline{\Omega'}$; thus $d_0$ is an invariant of the equivalence class of $U$ in 
$\calL_{r_0,r_2}$.
Take $E_0'\in SL_{r_0}(\Ok/\calP)$ so that with 
$$(\overline y_1\,\cdots\,\overline y_{d_0}\,\overline0\,\cdots\,\overline0)=(\overline y_1'\,\cdots\,\overline y_{r_0}')E_0',$$
$(\overline y_1,\ldots,\overline y_{d_0})$ is a basis for $\overline{\Omega'}$.
Now take $E_0\in SL_{r_0}(\Ok)$ so that $E_0\equiv E_0'\ (\calP)$ and set
$$(y_1\,\cdots\,y_{d_0}\,y_{d_0+1}''\,\cdots\,y_{r_0}'')=(y_1'\,\cdots\,y_{r_0}')E_0.$$

{\it Step 2.}  For $x\in L$, now let $\overline x=x+\calP L$ and let $\overline{\Omega'\cap L}$ denote the image of $\Omega'\cap L$ in $L/\calP L$.
Thus with $\pi\in\calP\smallsetminus \calP^2$ (as throughout this section),
$\{\overline{\pi y_1},\ldots,\overline{\pi y_{d_0}},\overline y_{d_0+1},\ldots,\overline y_{n-d_0-r_2}\}$ spans $\overline{\Omega'\cap L}$ with $\overline{\pi y_1},\ldots,\overline{\pi y_{d_0}}$ linearly independent.  
Let $d_0+d_1=\dim \overline{\Omega'\cap L}$; so $d_0$ and $d_1$ are invariants of 
the equivalence class of $U$ in 
$\calL_{r_0,r_2}$.
We extend $(\overline{\pi y_1},\ldots,\overline{\pi y_{d_0}})$ to an ordered basis
$$(\overline{\pi y_1},\ldots,\overline{\pi y_{d_0}},\overline y_{d_0+1},\ldots,\overline y_{d_0+d_1})$$ 
for $\overline{\Omega'\cap L}$ where $\overline y_{d_0+1},\ldots,\overline y_{d_0+d_1}$ lie in the span of 
$$\{\overline y_{d_0+1}'',\ldots,\overline y_{r_0}'',\overline y_{r_0+1}',\ldots,\overline y_{n-r_2}'\}.$$
Choose $E_1'\in SL_{n-d_0-r_2}(\Ok/\calP)$ so that
$$(\overline y_{d_0+1}\,\cdots\,\overline y_{d_0+d_1}\,\overline 0\,\cdots\,\overline0)
=(\overline y_{d_0+1}''\,\cdots\,\overline y_{r_0}''\,\overline y_{r_0+1}',\cdots,\overline y_{n-r_2}')E_1'.$$
Take $E_1\in SL_{n-d_0-r_2}(\Ok)$ so that $E_1\equiv E_1'\ (\calP)$, and set
$$E=\big<E_0,I_{n-r_0}\big>\big<I_{d_0},E_1,I_{r_2}\big>,\ 
\underline y=(y_1\,\cdots\,y_n)=(x_1\,\cdots\,x_m)UE,$$
and set $U_{\underline y}=UE.$
Hence we have 
$$y_i\in\begin{cases}\calP^{-1}L\smallsetminus L&\text{if $1\le i\le d_0$,}\\
L&\text{if $d_0<i\le d_0+d_1$,}\\
\calP L&\text{if $d_0+d_1<i\le n$.}
\end{cases}$$
%Also, it is straightforward to check that $y_1,\ldots,y_{d_0+d_1}$ are linearly independent.
We claim that $y_1,\ldots,y_{d_0+d_1}$ are linearly independent.  To see this,
 suppose that $u_1y_1+\cdots+u_{d_0+d_1}y_{d_0+d_1}=0,$ $u_i\in\kfld$, not all 0.
 Thus this equality holds over $\kfld_{\calP}$, and so multiplying by a suitable power of $\pi$,
 we can assume that all the $u_i\in\Ok_{\calP}$ with at least one $u_i$ a unit in $\Ok_{\calP}$.
 Thus in $\calP^{-1}L/L$, we have 
 $\overline{u_1y_1}+\cdots+\overline{u_{d_0}y_{d_0}}=\overline0,$
 and hence $u_1,\ldots,u_{d_0}\in\calP$ since $\overline y_1,\ldots,\overline y_{d_0}$ are
 linearly independent in $\calP^{-1}L/L$.  Thus we rewrite $u_i$ as $\pi u_i'$ for $1\le i\le d_0$.
 So in $L/\calP L$, we have
 $\overline u_1'=\cdots=\overline u_{d_0}'=\overline u_{d_0+1}=\cdots
 =\overline u_{d_0+d_1}=\overline 0$ by the choice of basis for $\overline{\Omega'\cap L}$.
 But this shows that all the $u_i$ lie in $\calP$, a contradiction.

(b) First suppose $\Lambda$ is a lattice with $\calP\Omega\subseteq\Lambda\subseteq\Delta$ and $\mult_{\{\Omega:\Lambda\}}(\calP)=r_0$,
$\mult_{\{\Omega:\Lambda\}}(\calP^{-1})=r_2$.
Choose $G'\in SL_n(\Ok_{\calP})$ so that with $(z_1'\,\cdots\,z_n')=(y_1\,\cdots\,y_n)G',$ we have
\begin{align*}
\Ok_{\calP}\Lambda
&=\calP\Ok_{\calP}z_1'\oplus\cdots\oplus\calP\Ok_{\calP}z_{r_0}'
\oplus\Ok_{\calP}z_{r_0+1}'\oplus\cdots\oplus\Ok_{\calP}z_{n-r_2}'\\
&\quad \oplus\calP^{-1}\Ok_{\calP}z_{n-r_2+1}'\oplus\cdots
\oplus\calP^{-1}\Ok_{\calP}z_n'.
\end{align*}
Now choose $G\in SL_n(\Ok)$ so that $G\equiv G'\ (\calP^2\Ok_{\calP})$.
Set $(z_1\,\cdots\,z_n)=(y_1\,\cdots\,y_n)G$.  Thus
$\Omega=\Ok z_1\oplus\cdots\oplus\Ok z_n$; set
\begin{align*}
\Lambda'
&=\calP z_1\oplus\cdots\oplus\calP z_{r_0}
\oplus\Ok z_{r_0+1}\oplus\cdots\oplus\Ok z_{n-r_2}\\
&\quad \oplus\calP^{-1}z_{n-r_2+1}\oplus\cdots\oplus
\calP^{-1}z_n.
\end{align*}
An easy check shows that $\Ok_{\calP}\Lambda'=\Ok_{\calP}\Lambda$;
for any prime $\calP'\not=\calP$, we have $\Ok_{\calP'}\Lambda'=\Ok_{\calP'}\Omega=\Ok_{\calP'}\Lambda$.  Thus $\Lambda'=\Lambda$.  Also, since $\Lambda\subseteq\Delta$,
one easily checks that $U_{\underline y}G\in\calL_{r_0,r_2}.$  Thus $\Lambda$ corresponds to an element in the equivalence class of $U_{\underline y}$.

Now take $E,G\in GL_n(\Ok)$ so that $U_{\underline y}E, U_{\underline y}G\in\calL_{r_0,r_2}$.  Set 
$$(w_1\,\cdots\,w_n)=(y_1\,\cdots\,y_n)E,\ 
(z_1\,\cdots\,z_n)=(y_1\,\cdots\,y_n)G,$$
\begin{align*}
\Lambda_E
&=\calP w_1\oplus\cdots\oplus\calP w_{r_0}
\oplus\Ok w_{r_0+1}\oplus\cdots\oplus\Ok w_{n-r_2}\\
&\quad \oplus\calP^{-1}w_{n-r_2+1}\oplus\cdots\oplus
\calP^{-1}w_n,\\
\Lambda_G
&=\calP z_1\oplus\cdots\oplus\calP z_{r_0}
\oplus\Ok z_{r_0+1}\oplus\cdots\oplus\Ok z_{n-r_2}\\
&\quad \oplus\calP^{-1}z_{n-r_2+1}\oplus\cdots\oplus
\calP^{-1}z_n.
\end{align*}
Using that $(z_1\,\cdots\,z_n)=(w_1\,\cdots\,w_n)E^{-1}G,$
one easily checks that $\Lambda_E=\Lambda_G$ if and only if $E^{-1}G\in\K_{r_0,r_2}.$  This proves (b).

(c) This claim follows easily from (a) and (b).
\end{proof}

We are almost ready to more precisely describe
$\theta(L;\tau)|S_1\cdots S_j(\calP^{-1}) T_j(\calP^2)$, in anology with Proposition 1.4 of [theta I].
But first, we need to define some more notation.

\smallskip
\noindent{\bf Definitions.}  
Fix non-negative integers $r_0,r_2$ so that $r_0+r_2\le j$.
Fix a reduced representative $U\in\calL_{r_0,r_2}$; set $\Omega=\Omega(U)$ and
$\underline y=(x_1\,\cdots\,x_m)U$ (so $\underline y$ is a reduced basis for $\Omega$).
\begin{enumerate}
\item[(a)]  We define an exponential sum associated to $\Omega$ by
$$\e\{\Omega,\tau\}=\sum_{C}\e\{Q[UC]\tau\}$$
where $C$ varies over $GL_n(\Ok).$
\item[(b)]  We say that $\Omega$ is even $\n$-integral if, for every $z,z'\in\Omega$, we have
$\q(z)\in2\n$ and $B_{\q}(z,z')\in\n$.  Note that when $\Omega\in\calL_{r_0,r_2}$ with
$\Omega\sim\Omega'$, $\Omega'$ is even $\n$-integral if and only if $\Omega$ is even $\n$-integral.
\item[(c)]  Suppose $\Omega$ is even $\n$-integral,
$\Delta$ is the formal intersection of $\Omega$ and $L$ (as defined in Proposition 4.2),
 and $\Lambda$ is a lattice so that
$\calP\Omega\subseteq\Lambda\subseteq\Delta$ with
$\mult_{\{\Omega:\Lambda\}}(\calP)=r_0$, $\mult_{\{\Omega:\Lambda\}}(\calP^{-1})=r_2$.
Take $G_{\underline y,\Lambda}$ as defined in Proposition 4.2.  Set $r_1=j-r_0-r_2$.
For each lattice $\Lambda_1\subseteq\Lambda$ so that $\Lambda_1+\calP(\Omega+\Lambda)$ has
dimension $r_1$ in $(\Omega\cap\Lambda)/\calP(\Omega+\Lambda)$, fix $U_{\Lambda_1}\in \Ok^{m,r_1}$
so that
$$\Lambda_1+\calP(\Omega+\Lambda)=(\Ok v_1\oplus\cdots\oplus\Ok v_{r_1})+\calP(\Omega+\Lambda)$$
where $(v_1\,\cdots\,v_{r_1})=(x_1\,\cdots\,x_m)U_{\Lambda_1}.$
Set
$$\balpha'_j(\Omega,\Lambda)=\sum_{\Lambda_1,W'}\e\{Q[U_{\Lambda_1}]W'\}$$
where $\pi\eta\delta W'$ varies over $(\Ok^{r_1,r_1}_{\sym}/\calP \Ok^{r_1,r_1}_{\sym})^{\times}.$
Note that when $r_0+r_2=j$, $\balpha'_j(\Omega,\Lambda)=1$, and 
when $r_0+r_2>j$, $\balpha'_j(\Omega,\Lambda)=0.$
\end{enumerate}

\smallskip

\begin{prop} 
Fix non-negative integers $r_0, r_2$ so that $r_0+r_2\le j$.  Then with $Y', G$ varying as in Proposition 4.1,
we have
\begin{align*}
&\sum_{Y',G} \chi_L(\det G)
\theta(\calL_{r_0,r_2};\tau)|\begin{pmatrix}G^{-1}&Y'\,^tG\\&^tG\end{pmatrix}\\
&\quad=
N(\calP)^{r_0(n-r_2+1)}\sum_{\underline y, \Lambda}
\balpha'_j(\Omega(\underline y),\Lambda)\, \e\{\Omega(\underline y), \tau\}
\end{align*}
where $\underline y$ varies so that
$U_{\underline y}$ varies over a representatives for the
equivalence classes in $\calL_{r_0,r_2}$ where $\Omega(\underline y)$ is even $\n $-integral; $\Lambda$ varies over all lattices so that $\calP\Omega(\underline y)\subseteq \Lambda\subseteq\Delta(\underline y)$ with 
$\mult_{\{\Omega(\underline y),\Lambda\}}(\calP)=r_0$ and
$\mult_{\{\Omega(\underline y),\Lambda\}}(\calP^{-1})=r_2.$
(Here $\Delta(\underline y)$ is the formal intersection of $\calP^{-1}\Omega(\underline y)$ and $L$, as defined in Proposition 4.2.)
Further, with $P_{r_2}$ as in Proposition 4.1, we have
\begin{align*}
&\chi_L(\det P'_{r_2})\sum_{Y',G} \chi_L(\det G)
\theta(\calL_{r_0,r_2};\tau)|\begin{pmatrix}G^{-1}&Y'\,^tG\\&^tG\end{pmatrix}|P_{r_2}\\
&\quad=
\sum_{Y',G} \chi_L(\det G)
\theta(\calL_{r_0,r_2};\tau)|\begin{pmatrix}G^{-1}&Y'\,^tG\\&^tG\end{pmatrix}.
\end{align*}
\end{prop}

\begin{proof}  We let 
$Y'=W+Y$
and $G=G_0G_1$ vary as described in Proposition 4.1.
We let 
$U_{\underline y}$ vary over a set of representatives for the equivalence classes in $\calL_{r_0,r_2}$, 
$E$ over $\K_{r_0,r_2}$; we let $\Lambda$, $G_{\underline y,\Lambda}$ vary as in Proposition 4.2.
We have $G_1\in\K_{r_0,r_2}$, so summing first on $G_0, G_1$ and then on $Y, W, \underline y, \Lambda, E$, we can replace $E$ by $EG_1$ and $Y$ by $G_1Y\,^tG_1$.  Then we can rearrange the order of summation to get
\begin{align*}
&\sum_{Y',G} \chi_L(\det G)
\theta(\calL_{r_0,r_2};\tau)|\begin{pmatrix}G^{-1}&Y'\,^tG\\&^tG\end{pmatrix}\\
&\quad=\sum_{\underline y,\Lambda}
\sum_{E,G_0}\e\{Q[U_{\underline y}G_{\underline y,\Lambda}EG_0^{-1}]\tau\}\\
&\qquad\cdot
\sum_{G_1,W}\e\{Q[U_{\underline y}G_{\underline y,\Lambda}EG_1]W\}
\sum_Y\e\{Q[U_{\underline y}G_{\underline y,\Lambda}EG_1]Y\}.
\end{align*}
The sum on $Y$ tests whether 
$\Omega(\underline y)$ is even $\n$-integral, returning 
$N(\calP)^{r_0(n-r_2+1)}$ if the answer is yes, and returning 0 otherwise.
So now suppose $\Omega(\underline y)$ is even $\n$-integral; fix
the parameter $\Lambda$ and consider
the sum on $G_1,W$.  As $E\in\K_{r_0,r_2}$,
with 
$$(z_1\,\cdots\,z_n)=(x_1\,\cdots\,x_m)U_{\underline y}G_{\underline y,\Lambda}E,$$
we have $\Omega(\underline y)=\Ok z_1\oplus\cdots\oplus\Ok z_n$ and
$$\Lambda=\calP z_1\oplus\cdots\oplus\calP z_{r_0}\oplus \Ok z_{r_0+1}\oplus\cdots\oplus
\Ok z_{n-r_2}\oplus\calP^{-1} z_{n-r_2+1}\oplus\cdots\oplus\calP^{-1} z_n.$$
Thus the sum on $G_1,W$ is $\balpha'_j(\Omega(\underline y),\Lambda).$
For fixed $\underline y,\Lambda$, we have 
$$\Omega(\underline y)=\Omega(U_{\underline y})=\Omega(U_{\underline y}G_{\underline y,\Lambda})$$
and as $E, G_0$ vary, $EG_0^{-1}$ varies over $GL_n(\Ok)$; thus the sum on $E,G_0$
is $\e\{\Omega(\underline y),\tau\}$.

To prove the final claim, we simply note that $\chi_L(\det P'_{r_2}) (\det P'_{r_2})^k=1$,
$\e\{Q[U_{\underline y}C]P'_{r_2}\tau\,^tP'_{r_2}\}=\e\{Q[U_{\underline y}CP'_{r_2}]\tau\},$
and $CP'_{r_2}$ varies over $GL_n(\Ok)$ as $C$ does.
\end{proof}

Next, we combine Propositions 4.1 and 4.3.

\begin{thm}  
We have
\begin{align*}
&\theta(L;\tau)|S_1\cdots S_j(\calP^{-1})T_j(\calP^2)\\
&\quad=
\sum_{\Omega,\Lambda}N(\calP)^{E'_j(\Omega,\Lambda)}\chi_L^*(\calP)^{e'_j(\Omega,\Lambda)}
\balpha'_j(\Omega,\Lambda)\e\{\Omega,\tau\}
\end{align*}
where $\Omega$ varies over all sublattices of $\calP^{-1}L$ with formal rank $n$, $\Lambda$ varies over all lattices so that $\calP\Omega\subseteq\Lambda\subseteq\Delta$ ($\Delta$ the formal intersection of $\calP^{-1}\Omega$ and $L$), and with $r_0=\mult_{\{\Omega:\Lambda\}}(\calP)$,
$r_2=\mult_{\{\Omega:\Lambda\}}(\calP^{-1})$, we have
$$E_j'(\Omega,\Lambda)=k(r_2-r_0-j)+r_0(n-r_2+1),\ e_j'(\Omega,\Lambda)=r_2-r_0-j.$$
\end{thm}

\begin{proof}  
Suppose that $\Omega$ is a sublattice of $\calP^{-1}L$ with formal rank $n$, and take $y_1,\ldots,y_n\in\calP^{-1}L$ so that $\Omega=\Ok y_1\oplus\cdots\oplus\Ok y_n$.  By Proposition 4.2,
we can assume that for some $d_0,d_1$, we have $y_1,\ldots,y_{d_0}\in\calP^{-1}L\smallsetminus L$,
$y_{d_0+1},\ldots,y_{d_0+d_1}\in L\smallsetminus\calP L$, $y_{d_0+d_1+1},\ldots,y_n\in\calP L$
with $y_1,\ldots,y_{d_0+d_1}$ linearly independent in $\kfld L$.  Let $\Delta$ be the formal intersection of $\calP^{-1}\Omega$ and $L$, and suppose $\Lambda$ is a lattice with $\calP\Omega\subseteq\Lambda\subseteq\Delta.$  Then with $r_0=\mult_{\{\Omega:\Lambda\}}(\calP)$,
$r_2=\mult_{\{\Omega:\Lambda\}}(\calP^{-1})$, we have $r_0\ge d_0$ and $r_2\le n-d_0-d_1$; hence
with $U_{\underline y}$ the matrix so that $(y_1\,\cdots\, y_n)=(x_1\,\cdots\,x_m)U_{\underline y}$,
we have $U_{\underline y}\in\calL_{r_0,r_2}.$  If $r_0+r_2>j$ then $\balpha_j'(\Omega,\Lambda)=0$ as there are no dimension $j-r_0-r_2$ subspaces of $(\Omega\cap\Lambda)/\calP(\Omega+\Lambda).$
The theorem now follows from Propositions 4.1 and 4.3.
\end{proof}

To complete the character sums in our description of the action of $T_j(\calP^2)$, we have the following.

\smallskip\noindent
{\bf Definitions.}  With $\calP$  a prime ideal, set
$$\widetilde T_j(\calP^2)=N(\calP)^{j(k-n-1)}
\sum_{\ell=0}^j \chi_L^*(\calP)^{\ell}N(\calP)^{k\ell} \bbeta(n-\ell,j-\ell) 
S_1\cdots S_{\ell}(\calP^{-1})
T_{\ell}(\calP^2)$$
where $T_0(\calP^2)$ is the identity map and
$$\bbeta(r,a)=\bbeta_{\calP}(r,a)=\prod_{i=0}^{a-1}\frac{(N(\calP)^{r-i}-1)}{(N(\calP)^{a-i}-1)}.$$
(So $\bbeta(r,a)$ is the number of dimension $a$ subpaces of a dimension $r$ space over 
$\Ok/\calP$.)  To ease notation, we set $\F=\Ok/\calP$.  Set  
$$\boldeta(r,a)=\boldeta_{\calP}(r,a)=\prod_{i=a}^{r-1}(N(\calP)^{r-i}-N(\calP)^{a-i});$$
so $\boldeta(r,0)=|GL_r(\F)|$ and for $1\le a\le r$ and $C\in\F^{r,a}$ with $\rank C=a$, $\boldeta(r,a)$ is the number of ways to extend $C$ to an element of $GL_r(\F)$.

With $\Omega$ a subspace of $\calP^{-1}L$ with formal rank $n$, and
$\Lambda$ a lattice so that
$\calP\Omega\subseteq\Lambda\subseteq\Delta$ (where $\Delta$ is the formal intersection of $\calP^{-1}\Omega$ and $L$), let $r_0=\mult_{\{\Omega:\Lambda\}}(\calP)$, 
$r_2=\mult_{\{\Omega:\Lambda\}}(\calP^{-1})$, $r=r_0+r_2$.
With $\Omega_1=(\Omega\cap\Lambda)/\calP(\Omega+\Lambda)$, we consider $\Omega_1$ as a quadratic space over $\F$, with symmetric bilinear form $B'=\eta'B_{\q}$ (modulo $\calP$); when $\calP\nmid2$, we take the quadratic form on $\Omega_1$ to be $\q'=\eta'\q$ (modulo $\calP$), and when $\calP|2$, we take $\q'=\frac{1}{2}\eta'\q$ (and then we have the relation $q'(x+y)=\q'(x)+\q'(y)+B'(x,y)$).  With $\Lambda_1$ a subspace of $\Omega_1$, we say $\Lambda_1$ is totally isotropic if $\q'(x)=0\in\F$ for every $x\in\Lambda_1$.

\smallskip

This gives us the analogue of the first step in the proof of Theorem 2.1 \cite{theta II}.

\begin{thm}
Suppose that $\calP$ is a prime ideal with $\calP\nmid \stufe$.
We have
\begin{align*}
&\theta(L;\tau)|\widetilde T_j(\calP^2)\\
&\quad =
\sum_{\Omega}
\left(\sum_{\calP\Omega\subseteq\Lambda\subseteq\Delta}\chi^*(\calP)^{e_j(\Omega,\Lambda)}N(\calP)^{E_j(\Omega,\Lambda)}\balpha_j(\Omega,\Lambda)\right)\e\{\Omega,\tau\}
\end{align*}
where the sum is over all even $\n$-integral sublattices
$\Omega$ of $\calP^{-1}L$ with formal rank $n$, and $\Delta$ is the formal intersection of $\calP^{-1}\Omega$ and $L$.
Also, with $r_0=\mult_{\{\Omega:\Lambda\}}(\calP)$ and
$r_2=\mult_{\{\Omega:\Lambda\}}(\calP^{-1})$, we take
$e_j(\Omega,\Lambda)=j+r_2-r_0$ and
$E_j(\Omega,\Lambda)=k(j+r_2-r_0)+r_0(n-r_2+1)+(j-r_0-r_2)(j-r_0-r_2+1)/2-j(n+1),$
and $\balpha_j(\Omega,\Lambda)$ is the number of totally isotropic, codimension $n-j$ subspaces of
$(\Omega\cap\Lambda)/\calP(\Omega+\Lambda)$ which has the quadratic form $\q'$ as defined above.
\end{thm}

\begin{proof}
Take $\Omega\subseteq\calP^{-1}L$ to have formal rank $n$, and fix $\Lambda$ so that
$\calP\Omega\subseteq\Lambda\subseteq\Delta$ where $\Delta$ is the formal intersection of $\calP^{-1}\Omega$ and $L$.  Set $r_0=\mult_{\{\Omega:\Lambda\}}(\calP)$, $r_2=\mult_{\{\Omega:\Lambda\}}(\calP^{-1})$, and $r=r_0+r_2$.
Assume that the quadratic form $\q$ restricted to $\Omega$ is even $\n$-integral.
Given Theorem 4.4 and the definition of $\widetilde T_j(\calP^2)$,
we want to evaluate
$$\sum_{\ell=0}^j \bbeta(n-\ell,j-\ell)\balpha_{\ell}'(\Omega,\Lambda).$$

First suppose that $\calP\nmid 2$; set $\F=\Ok/\calP$.  
Take $\Omega_1=(\Lambda\cap\Omega)/\calP(\Lambda+\Omega)$; as discussed above, we
consider $\Omega_1$ as a dimension $n-r$ vector space over $\F$ equipped with the quadratic form
$\q'=\eta'\q$.  Take $V\in\F^{n-r,n-r}_{\sym}$ so that $\Omega_1\simeq 2V$.  
In Theorem 4.1, we replace $Q[G_1]W$ by $2V[G_1']W_1\pi'\beta'$ where $G_1'$ varies as in Theorem 4.1, and $"_1$ varies over $\Ok^{\ell-r,\ell-r}_{\sym}$ modulo $\calP$ so that $\det W_1\not\in\calP$.  Then
$\Omega_1G'_1\big<I_{\ell-r},0_{n-\ell}\big>$ varies over all codimension $n-\ell$ subspaces of $\Omega_1$, and 
$$\balpha_{\ell}'(\Omega,\Lambda)
=\sum_{G'_1,W_1}\e\{2G'_1V\,^tG'_1\big<W_1,0_{n-\ell}\big>\pi'\beta'\}.$$

For $U\in\F^{a,a}_{\sym}$, we set
$$r^*(V,U)=\#\{C\in\F^{r,a}:\ ^tCVC=U,\ \rank C=a\ \}$$
and we set $R^*(V,U)=r^*(V,U)/o(U)$ where $O(U)$ is the  orthogonal group of $U$ and $o(U)=|O(U)|$.  
(The condition that $\rank C=a$ ensures that $U$ describes the quadratic form on a rank $a$ space.)
% Then
% $$\sum_{\ell=0}^j\bbeta(n-\ell,j-\ell)\balpha'_{\ell}(\Omega,\Lambda)
% =N(\calP)^{(j-r)(j-r+1)/2}R^*(V,0_{j-r})$$
% where $0_{j-r}$ denotes the $(j-r)\times(j-r)$ matrix of zeros.
%Let $a$ be an integer with $0\le a\le j-r$.

Now, given a dimension $j'=j-r$ subspace $\Omega_1'$ of $\Omega_1$, and a dimension $a$ subspace $\Lambda_1$ of $\Omega_1'$
($0\le a\le j'$), there are matrices $V'\in\F^{j',j'}_{\sym}$ and $U\in\F^{a,a}_{\sym}$ so that $\Omega_1'\simeq 2V'$ and $\Lambda_1\simeq 2U$.  Also, for given $\Lambda_1$ as above,
there are $\bbeta(n-r-a,j'-a)$ dimension $j'$ subspaces $\Omega_1'$ of $\Omega_1$ that contain $\Lambda_1$.
(Note that if $a=0$ then $\Lambda_1=\{0\}$.)  Letting $\Lambda_1$ vary over all dimension $a$ subspaces of $\Omega_1$ and sorting $\Lambda_1$ and $\Omega_1'$ by isometry classes,
we get
\begin{align*}
&\sum_{\cls U\in\F^{a,a}_{\sym}} \bbeta(n-r-a,j-r-a)R^*(V,U)\\
&\quad =
\sum_{\cls U\in\F^{a,a}_{\sym}}\sum_{\cls V'\in\F^{j-r,j-r}_{\sym} } R^*(V,V')R^*(V',U).
\end{align*}
Thus
\begin{align*}
&\sum_{a=0}^{j-r} \bbeta(n-r-a,j-r-a)\balpha'_{\ell}(\Omega,\Lambda)\\
&\quad=
\sum_{a=0}^{j-r} \bbeta(n-r-a,j-r-a)
\sum_{\cls U\in\F^{a,a}_{\sym}}R^*(V,U)
\sum_{\substack{W_1\in\F^{a,a}_{\sym}\\ \det W_1\not=0}} \e\{2\pi'\beta'UW_1\}\\
&\quad=
\sum_{\cls V'\in\F^{j-r,j-r}_{\sym}} R^*(V,V')
\sum_{a=0}^{j-r}\sum_{\cls U\in\F^{a,a}_{\sym}} R^*(V',U)
\sum_{\substack{W_1\in\F^{a,a}_{\sym}\\ \det W_1\not=0}}
\e\{2\pi'\beta'UW_1\}.
\end{align*}

We now fix $V'\in\F^{j-r,j-r}_{\sym}$, and to ease notation, we set $j'=j-r$.
We want to show that 
\begin{align*}
&\sum_{a=0}^{j-r}\sum_{\cls U\in\F^{a,a}_{\sym}} R^*(V',U)
\sum_{\substack{W'\in\F^{a,a}_{\sym}\\ \det W'\not=0}}
\e\{2\pi'\beta'UW'\}\\
&\quad=
\sum_{Y\in\F^{j',j'}_{\sym}}\e\{2V'Y\pi'\beta'\}.
\end{align*}
We know that $GL_{j'}(\F)$ acts by conjugation on $\F^{j',j'}_{\sym}$, and thus $\F^{j',j'}_{\sym}$ is partitioned into orbits (i.e. isometry classes) with representatives
$$0_{j'},\ \big<I_{c},0_{j'-c}\big>,\ \big<J_{c},0_{j'-c}\big>,\ 1\le c\le j'$$
where $J_c=\big<I_{c-1},\omega\big>$, and $\omega\in\Ok$ with $\left(\frac{\omega}{\calP}\right)=-1$
(this can be deduced, for instance, from 92:1 of \cite{O'M}).
Also, with $U\in\F^{j',j'}_{\sym}$,
as $G$ varies over $GL_{j'}(F)$, $^tGUG$ varies $o(U)$ times over
the distinct matrices in the isometry class of $U$.
Thus we get
\begin{align*}
	&\sum_{Y\in\F^{j',j'}_{\sym}}\e\{2V'Y\pi'\beta'\}\\
	&\quad=1+
	\sum_{a=1}^{j'}\sum_{G\in GL_{j'}(\F)}
	\frac{1}{o(I_a\perp 0_{j'-a})}
	\e\{2V'\,^tG\big<I_a,0_{j'-a}\big>G\pi'\beta'\}\\
	&\qquad +
	\sum_{a=1}^{j'}\sum_{G\in GL_{j'}(\F)}
	\frac{1}{o(J_a\perp 0_{j'-a})}
	\e\{2V'\,^tG\big<J_a,0_{j'-a}\big>G\pi'\beta'\}.
\end{align*}
Now fix $a$, $1\le a\le j'$.  With $U_G$ the upper left $a\times a$ block of $GV'\,^tG$, we have
\begin{align*}
&\sum_{G\in GL_{j'}(\F)} \e\{2V'\,^tG\big<I_a,0_{j'-a}\big>G\pi'\beta'\}\\
&\quad=
\sum_{G\in GL_{j'}(\F)} \e\{2GV'\,^tG\big<I_a,0_{j'-a}\big>\pi'\beta'\}\\
&\quad=
\sum_{G\in GL_{j'}(\F)} \e\{2U_GI_a\pi'\beta'\}.
\end{align*}
Given $U\in\F^{a,a}_{\sym}$, the number of $G\in GL_{j'}(\F)$ so that $U_G=U$ is
\begin{align*}
\#\{C\in\F^{j',a}:\ ^tCV'C=U,\ \rank C=a\ \}
%&\quad=
=r^*(V',U)\boldeta(j',a)
\end{align*}
as the number of ways to extend $C\in\F^{j',a}$ to an element of $GL_{j'}(\F)$ is 0 if $\rank C<a$, and $\boldeta(j',a)$ otherwise.
%$$\begin{cases}0&\text{if $\rank C<a$,}\\
%\boldeta(j',a)&\text{otherwise.}
%\end{cases}$$
So (for $a$ still fixed),
\begin{align*}
&\sum_{G\in GL_{j'}(\F)}\frac{1}{o(I_a\perp 0_{j'-a})}
\e\{2V'\,^tG\big<I_a,0_{j'-a}\big>G\pi'\beta'\}\\
&\quad=
\sum_{U\in\F^{a,a}_{\sym}}
\frac{\boldeta(j',a)}{o(I_a\perp 0_{j'-a})}	
	r^*(V',U)\e\{2U\pi'\beta'\}.
\end{align*}
We now sort the $U$ in the above sum into isometry classes, and note that for $U_0\in\F^{a,a}_{\sym}$,
\begin{align*}
&\sum_{U\in\cls U_0}r^*(V',U)\e\{2U\pi'\beta'\}\\
&\quad=
\sum_{G\in O(U_0)\backslash GL_a(\F)}
r^*(V',U_0)\e\{2\,^tGU_0G\pi'\beta'\}\\
&\quad=
\sum_{G\in GL_a(\F)} R^*(V,U_0)\e\{2\,^tGU_0G\pi'\beta'\}.
\end{align*}
It is also not difficult to check the
$o(I_a\perp 0_{j',a})=o(I_a)\boldeta(j',a)$ so
\begin{align*}
&\sum_{U\in\F^{a,a}_{\sym}}
\frac{\boldeta(j',a)}{o(I_a\perp 0_{j'-a})} r^*(V',U)\e\{2U\pi'\beta'\}\\
&\quad=
\sum_{cls U\in\F^{a,a}_{\sym}}
\frac{R^*(V',U)}{o(I_a)} \sum_{G\in GL_a(\F)} 
\e\{2UGI_a\,^tG\pi'\beta'\}\\
&\quad=
\sum_{cls U\in\F^{a,a}_{\sym}}
R^*(V',U)\sum_{W'\in\cls I_a} \e\{2UW'\pi'\beta'\}.
\end{align*}

Similar arguments hold when $I_a$ is replaced by $J_a$, giving us
\begin{align*}
\sum_{Y\in\F^{j',j'}_{\sym}} \e\{2V'Y\pi'\beta'\}
&=
1+\sum_{a=1}^{j'}\sum_{\cls U\in\F^{a,a}_{\sym}}
R^*(V',U)\sum_{W_1\in\cls I_a}\e\{2UW_1\pi'\beta'\}\\
&\quad+
\sum_{a=1}^{j'}\sum_{\cls U\in\F^{a,a}_{\sym}}
R^*(V',U)\sum_{W_1\in\cls J_a}\e\{2UW_1\pi'\beta'\}\\
&=
\sum_{a=0}^{j'}\sum_{\cls U\F^{a,a}_{\sym}} R^*(V',U)
\sum_{\substack{W_1\in\F^{a,a}_{\sym}\\ \det W_1\not=0}}
\e\{2UW_1\pi'\beta'\}.
\end{align*}
The sum $\sum_{Y\in\F^{j',j'}_{\sym}} \e\{2V'Y\pi'\beta'\}$
tests whether $V'$ is 0 modulo $\calP$, returning $N(\calP)^{j'(j'+1)/2}$ if $V'$ passes this test and 0 otherwise.  So
remembering that $j'=j-r$, we see that
$$\sum_{a=0}^{j-r}\bbeta(n-r-a,j-r-a)\balpha_{\ell}'(\Omega,\Lambda)
=N(\calP)^{(j-r)(j-r+1)/2}R^*(V,0_{j-r}).$$
Also recall that $V$ is chosen so that
 $\Omega_1\simeq 2V$ where $V$ is determined by $\Omega$ and $\Lambda$ as described at the beginning of the proof; so
 $\balpha_j(\Omega,\Lambda)=R^*(V,0_{j-r})$, and the theorem follows in the case that $\calP\nmid 2$.
 
 Now suppose $\calP|2$.  
 Then the above argument carries over almost directly, with just a few adjustments.  First, with $\Omega_1$ as above and $\q'=\frac{1}{2}\eta'\q$, we take $V\in\Ok^{n-r,n-r}_{\sym}$ 
 to encode $\q'$ as follows.  With 
 $\Omega_1=\F x_1\oplus\cdots\oplus\F x_{n-r}$, the $s,t$-entry of $V$ is $B'(x_s,x_t)$ when $s\not=t$, and $2\q'(x_s)$ when $s=t$.  With $G,G'\in\Ok^{n-r,n-r}$ so that $G\equiv G'\ (\calP)$ and $\det G\not\in\calP$, the diagonal of
 $GV\,^tG$ is congruent modulo $2\calP$ to the diagonal of $G'V\,^tG'$.  So for $G\in GL_{n-r}(\F)$, we have
 $G\in O(V)$ if $GV\,^tG$ and $V$ are congruent modulo $\calP$ with the diagonal of $GV\,^tG$ and the diagonal of $V$ congruent modulo $2\calP$.  With $\Omega_1'$ a subspace of $\Omega_1$ of dimension $j'$, we take $V'$ to be a matrix encoding $\q'$ on $\Omega_1'$, and $r^*(V,V')$ is the number of matrices $C\in\F^{n-r,j'}$ with rank $j'$ so that $^tCVC\equiv V'\ (\calP)$ with the diagonal of
 $^tCVC$ congruent modulo $2\calP$ to the diagonal of $V'$.
 Then the above argument when $calP\nmid 2$ carries over to the case when $\calP|2$.
 \end{proof}

\bigskip

\section{Average theta series as Hecke eigenforms}
\smallskip

In this section, we adapt the local arguments from \cite{theta I} and \cite{theta II} to Hilbert-Siegel theta series.
As a first step, we adapt Proposition 2.1 from \cite{theta II}.  Throughout, $\calP$ is a prime ideal with
$\calP\nmid\stufe$.

\smallskip\noindent
{\bf Definitions.}  For $r,m\in\Z$ with $r>0$, set
$$\bdelta(m,r)=\prod_{i=0}^{r-1}(N(\calP)^{m-i}+1),\ 
\bmu(m,r)=\prod_{i=0}^{r-1}(N(\calP)^{m-i}-1).$$
We take $\bdelta(m,0)=1=\bmu(m,0).$
\smallskip

\begin{prop}  
Take $j\in\Z_+$ so that $j\le n$.
%, and take $\calP$ to be a prime ideal with $\calP\nmid\stufe$.
% ; also, assume that $j\le k$ if $\chi^*(\calP)=1$, and $j<k$ if $\chi^*(\calP)=-1$.
We have
$$\theta(L;\tau)|\widetilde T_j(\calP^2)=\sum_{\Omega}\widetilde c_j(\Omega)\e\{\Omega,\tau\}$$
where $\Omega$ varies over all even $\eta$-integral free sublattices of $\calP^{-1}L$ with formal rank $n$, and $\widetilde c_j(\Omega)$ is defined as follows.  First, as in Proposition 4.2, we decompose $\Omega$ as
$\Ok y_1\oplus\cdots\oplus\Ok y_n$ where
$$y_i\in\begin{cases}\calP^{-1}L\smallsetminus L&\text{if $1\le i\le d_0$,}\\
L\smallsetminus\calP L&\text{if $d_0<i\le d_0+d_1$,}\\
\calP L&\text{otherwise}\end{cases}$$
with $y_1,\ldots,y_{d_0+d_1}$ linearly independent in $\kfld L$
(so $d_0,d_1$ are invariants of $\Omega$).
Set $d_2=n-d_0-d_1$.  
With $\Omega_1=\Ok y_{d_0+1}\oplus\cdots\oplus\Ok y_{d_0+d_1}$,
let $\overline\Omega_1$ denote the space
$\Omega_1/\calP\Omega_1$ equipped with the quadratic form $\eta'\q$.  Set
$$E=E'(\ell,t,\Omega)=\ell(k-d_0-d_1)+\ell(\ell-1)/2+t(k-n)+t(t+1)/2.$$
\begin{enumerate}
\item[(a)]  Say $\chi^*(\calP)=1.$ Then
\begin{align*}
\widetilde c_j(\Omega)
&=\sum_{\ell,t}N(\calP)^E\varphi_{\ell}(\overline\Omega_1)
\bdelta(k-d_0-\ell-1,t)\bbeta(d_2,t)\\
&\quad\cdot
\bbeta(n-d_0-\ell-t,j-d_0-\ell-t).
\end{align*}
\item[(b)]  Say $\chi^*(\calP)=-1.$ Then
\begin{align*}
\widetilde c_j(\Omega)
&=\sum_{\ell,t}(-1)^{\ell}N(\calP)^E\varphi_{\ell}(\overline\Omega_1)
\bbeta(k-d_0-\ell-1,t)\bmu(d_2,t)\\
&\quad\cdot\bbeta(n-d_0-\ell-t,j-d_0-\ell-t).
\end{align*}
\end{enumerate}

\end{prop}

\begin{proof}
Beginning with Theorem 4.5, the proof proceeds exactly as the proof of Proposition 1.4 \cite{theta I} and of Proposition 2.1  \cite{theta II}.  This consists of constructing all lattices $\Lambda$ with
$\calP\Omega\subseteq\Lambda\subseteq\Delta$ where $\Delta$ is the formal intersection of $\calP^{-1}\Omega$ and $L$.  The construction is in two stages; in each stage, we work over an $\Ok/\calP$-space.  Thus the theory is essentially the same as when we work over $\Z/p\Z$.  The theory of quadratic forms over finite fields that we use can be found in \cite{Ger} and \cite{O'M}.
\end{proof}

Next we have the analogue of Proposition 1.5 \cite{theta I} and Proposition 2.2 \cite{theta II}.

\begin{prop}
Take $j\in\Z_+$ so that $j\le n$; also, assume that $j\le k$ if $\chi^*(\calP)=1$, and $j<k$ if $\chi^*(\calP)=-1$.
We let $K_j$ vary over all lattices so that $\calP L\subseteq K_j\subseteq \calP^{-1}L$,
$\mult_{\{L:K_j\}}(\calP^{-1})=\mult_{\{L:K_j\}}(\calP)=j$, and $K_j\in\gen L.$
Then
$$\sum_{K_j}\theta(K_j;\tau)=\sum_{\Omega}b_j(\Omega)\e\{\Omega,\tau\}$$
where $\Omega$ varies over all even $\n$-integral free sublattices of $\calP^{-1}L$ with formal rank $n$, and decomposing $\Omega$ is in Proposition 5.1, we have invariants $d_0,d_1,d_2$ attached to $\Omega$; then
if $\chi^*(\calP)=1$, we have
\begin{align*}
b_j(\Omega)&=N(\calP)^{(j-d_0)(j-d_0-1)/2}\sum_{\ell}N(\calP)^{\ell(k-j-d_1+\ell)}
\varphi_{\ell}(\overline\Omega_1)\\
&\quad\cdot
\bdelta(k-d_0-\ell-1,j-d_0-\ell)\bbeta(k-d_0-d_1,j-d_0-\ell),
\end{align*}
and if $\chi^*(\calP)=-1$, we have
\begin{align*}
b_j(\Omega)&=N(\calP)^{(j-d_0)(j-d_0-1)/2}\sum_{\ell}(-1)^{\ell}N(\calP)^{\ell(k-j-d_1+\ell)}
\varphi_{\ell}(\overline\Omega_1)\\
&\quad\cdot
\bbeta(k-d_0-\ell-1,j-d_0-\ell)\bdelta(k-d_0-d_1,j-d_0-\ell).
\end{align*}
\end{prop}

\begin{proof}
Again, all the arguments are local.  Here we equip $L/\calP L$ with the quadratic form $\eta'\q$ so that $L/\calP L$ is a regular quadratic space over $\Ok/\calP$, with $L/\calP L$ hyperbolic if and only if $\chi^*(\calP)=1$.  Then for $\Omega$ as in the proposition, all the $K_j$ containing $\Omega$ are constructed.  An element needed to complete the proof is an analogue of Lemma 4.1 \cite{theta II}; again, the proof for this is completely local.
\end{proof}

Now Propositions 5.1 and 5.2 above, together with the proof of Theorem 2.3 \cite{theta II} gives us the following.

\begin{thm}  Take $j\in\Z_+$ so that $j\le n$; also, assume that $j\le k$ if $\chi^*(\calP)=1$, and $j<k$ if $\chi^*(\calP)=-1$.  Set
$$u_i(j)=(-1)^iN(\calP)^{i(i-1)/2}\bbeta(n-j+i,i),
\ T_j'(\calP^2)=\sum_{0\le i\le j}u_i(j)\widetilde T_{j-i}(\calP^2),$$
$$v_i(j)=
\begin{cases}(-1)^i\bbeta(k-n+i-1,i)\bdelta(k-j+i-1,i)&\text{if $\chi^*(\calP)=1$,}\\
(-1)^i\bdelta(k-n+i-1,i)\bbeta(k-j+i-1,i)&\text{if $\chi^*(\calP)=-1$}
\end{cases}$$
Then
$$\theta(L)|T'_j(\calP^2)=\sum_{0\le i\le j}v_i(j)\left(\sum_{K_{j-i}}\theta(K_{j-i})\right)$$
where $K_{j-i}$ varies subject to $\calP L\subseteq K_{j-i}\subseteq\calP^{-1}L,$
$$\mult_{\{L:K_{j-i}\}}(\calP^{-1})=\mult_{\{L:K_{j-i}\}}(\calP)=j-i,$$
and $K_{j-i}\in\gen L$.
\end{thm}

The final step is average over the genus of $L$; we now define the average theta series.

\smallskip\noindent
{\bf Definition.}  For $L'$ a lattice in the genus of $L$, let $o(L')$ denote the order of the orthogonal group of $L'$.  Define the average theta series attached to the genus of $L$ as
$$\theta(\gen L;\tau)=\sum_{\cls L'\in\gen L}\frac{1}{o(L')}\theta(L';\tau)$$
where $\cls L'$ denotes the isometry class of $L'$.
(Note that sometimes people normalise the average theta series by $\frac{1}{\mass L}$ where
$\mass L=\sum_{\cls L'\in\gen L}\frac{1}{o(L')}$.)
\smallskip

We now state our main result; 
the proof combines those of Corollary 2.4 and Theorem 3.3 \cite{theta II} and primarily
consists of elementary combinatorial arguments over a finite field.

\begin{cor}  Recall that $\calP$ is a prime ideal with $\calP\nmid\stufe$, and $1\le j\le n$.
\begin{enumerate}
\item[(a)]
Suppose that $j\le k$ if $\chi^*(\calP)=1$, and $j<k$ if $\chi^*(\calP)=-1$.
We have
$$\theta(\gen L)|T'_j(\calP^2)=\lambda_j(\calP^2)\theta(\gen L)$$
where
$$\lambda_j(\calP^2)=
\begin{cases}
N(\calP)^{j(k-n)+j(j-1)/2}\bbeta(n,j)\bdelta(k-1,j)&\text{if $\chi^*(\calP)=1$,}\\
N(\calP)^{j(k-n)+j(j-1)/2}\bbeta(n,j)\bmu(k-1,j)&\text{if $\chi^*(\calP)=-1$.}
\end{cases}$$
\item[(b)]  
Suppose that $j> k$ if $\chi^*(\calP)=1$, and $j\ge k$ if $\chi^*(\calP)=-1$.
Then $$\theta(\gen L)|T'_j(\calP^2)=0.$$
\end{enumerate}
\end{cor}


\bigskip

\begin{thebibliography}{0}



\bibitem{And} A.N. Andrianov, \emph{Quadratic Forms and Hecke Operators}, Springer-Verlag,  1987.

\bibitem{CW} S. Caulk, L.H. Walling, \emph{Hecke operators on Hilbert-Siegel modular
forms}, Int. J. Number Theory {\bf 3}, No. 3 (2007), 391--420.


\bibitem{Eichler}  M. Eichler, \emph{On theta functions of real algebraic
number fields},  Acta Arith. {\bf 33} (1977), 269-292.


\bibitem {Ger} L. Gerstein, \emph{Basic Quadratic Forms}, Graduate
Studies in Math. Vol. 90, Amer. Math. Soc., Providence, 2008.


\bibitem{O'M} O.T. O'Meara, \emph{Introduction to Quadratic Forms}, Springer, Berlin, 1987.

\bibitem{thesis} L.H. Walling
\emph{Hecke operators on theta
	series attached to lattices of arbitrary rank},
Acta Arith. {\bf 54} (1990), 213-240.

\bibitem{theta I}
L.H. Walling, \emph{Action of Hecke operators on Siegel theta series I},
International J. of Number Theory {\bf 2} (2006), 169-186.

\bibitem{theta II} L.H. Walling, \emph{Action of Hecke operators on Siegel theta series II},
International J. of Number Theory {\bf 4} (2008), 981-1008.

\bibitem{W eis series} L.H. Walling, \emph{Hecke eigenvalues and relations for Siegel Eisenstein series of arbitrary degree, level, and character},  International J. Number Theory {\bf 13} No. 2 (2017), 325-370.

\bibitem{theta eis}  L.H. Walling, \emph{Explicitly realizing average Siegel theta series as linear combinations of Eisenstein series},
Ramanujan J. Math. {\bf 47} (2018), 475-499.

\end{thebibliography}
\end{document}